\documentclass[11pt,reqno]{amsart}
\usepackage{amsmath,amssymb,amsfonts,amsthm, amscd,indentfirst}
\usepackage{amsmath,latexsym,amssymb,amsmath,%awheremscd,
amscd,amsthm,amsxtra}
\usepackage{hyperref}\usepackage{url}
\usepackage{color}
\usepackage{pdfpages}
\usepackage{graphics}
\usepackage{enumitem}

\usepackage{mathtools}

\usepackage{graphicx}
\usepackage{caption}

\usepackage{epsfig,here}
\usepackage{subfigure,here}

\newtheorem{defn}{Definition}[section]
\newtheorem{theorem}{Theorem}[section]
\newtheorem{prop}{Proposition}[section]
\newtheorem{lemma}[theorem]{Lemma} 
\newtheorem{corollary}[theorem]{Corollary} 

\newtheorem{remark}{Remark}[section]

\usepackage[top=1.2in,bottom=1.2in,left=1.0in,right=1.0in]{geometry} 

\hypersetup{colorlinks=true, linkcolor=black, citecolor=black}

\def\rr{\mathbb{R}}

\def\hh{\mathbb{H}}

\def\tr{\mathrm{tr}}

\def\p{\partial}

\def\th{\theta}

\def\p{\partial}

\def\<{\langle}
\def\>{\rangle}

\def\n{\nabla}

\numberwithin{equation} {section}

\begin{document}

\title[Quermassintegral inequalities for capillary hypersurfaces in a horoball] {Quermassintegral inequalities for capillary hypersurfaces in a horoball}% in hyperbolic space}
	\author{Jinyu Guo}
	\address{School of Mathematics and Statistics, FJKLMAA, Key Laboratory of Analytical Mathematics and Applications (Ministry of Education), Fujian Normal University, 350117, Fuzhou, China}
\email{guojinyu@fjnu.edu.cn}
\author{Liangjun Weng}
	\address{Centro di Ricerca Matematica Ennio De Giorgi, Scuola Normale Superiore,
56126, Pisa, Italy}
	\email{liangjun.weng@sns.it}

	\begin{abstract}
In this paper, we study the quermassintegral inequalities for capillary hypersurfaces in a horoball in hyperbolic space. We construct a locally constrained nonlinear curvature flow that preserves the \(k\)-th quermassintegral while monotonically increasing the \(l\)-th quermassintegral for any \(1\leq l<k\leq n\). Using this flow, we establish a full family of the quermassintegral inequalities for horo-convex capillary hypersurfaces in a horoball with contact angle $\theta\in(0,\pi/2]$. Moreover, we weaken the horo-convexity to convexity and obtain the quermassintegral inequalities for convex capillary hypersurfaces between the \(n\)-th  and \(l\)-th quermassintegrals.

\end{abstract}

  
  
	\maketitle
	
%	\medskip

%\tableofcontents

\section{Introduction}

%The Alexandrov--Fenchel inequality \cite{Alex, Fenchel} relates the mixed volumes of convex bodies in Euclidean space. It yields sharp inequalities between their quermassintegrals, including the isoperimetric and Minkowski inequalities. In this paper, we study analogous inequalities for convex capillary hypersurfaces in a horoball in hyperbolic space. We first recall the quermassintegrals of a convex body 

The Alexandrov--Fenchel inequality \cite{Alex, Fenchel} relates the mixed volumes of convex bodies in Euclidean space. It yields sharp inequalities between their quermassintegrals, including the classical isoperimetric and Minkowski inequalities. In this paper, we establish analogous inequalities for convex capillary hypersurfaces in a horoball in hyperbolic space. We first recall the quermassintegrals of a convex body $\Omega$ in hyperbolic space $\mathbb{H}^{n+1}$. For $1\leq j\leq n+1$, $W_j(\Omega)$ is the suitably normalized invariant measure of the set of $j$-dimensional totally geodesic subspaces intersecting $\Omega$; see \cite[Chapter 17]{Santalo}. We use the normalization as \begin{equation*} W_0(\Omega)=\operatorname{Vol}(\Omega), \qquad W_1(\Omega)=\frac{1}{n+1}\operatorname{Area}(\partial\Omega). \end{equation*} When $\partial\Omega$ is of class $C^2$, these functionals satisfy the recurrence relation \begin{equation*} W_{j+1}(\Omega) =\frac{1}{n+1}\int_{\partial\Omega}H_j dA  -\frac{j}{n+2-j}W_{j-1}(\Omega),\qquad 1\leq j\leq n,
\end{equation*}
where $H_j$ is the normalized $j$-th mean curvature of
$\partial\Omega$ %with respect to the outward unit normal,
and $H_0\equiv1$.

The first variation formula provides the link between
quermassintegrals and curvature flows. For a smooth family of
convex bodies $\Omega_t$ whose boundaries have outward normal
speed $f$, one has
\begin{equation*}
  \frac{d}{dt}W_j(\Omega_t)
  =\frac{n+1-j}{n+1}
   \int_{\partial\Omega_t} f H_j dA_t,
  \qquad 0\leq j\leq n,
\end{equation*}see, e.g., \cite[Proposition 3.1]{WX1}, or
\cite[Lemma 2.3]{ACW} after adjusting the normalization.

For the Minkowski formulae below, we write the hyperbolic metric
in geodesic polar coordinates about a fixed origin as
\begin{equation*}
  \bar g=dr^2+\lambda^2(r)g_{\mathbb{S}^n},
  \qquad
  \lambda(r)\coloneqq \sinh r,
  \qquad r>0,
\end{equation*}
where $g_{\mathbb{S}^n}$ is the standard round metric on
$\mathbb{S}^n$.

For a smooth closed oriented hypersurface $\Sigma\subset\mathbb{H}^{n+1}$, the conformal Killing field $\lambda\partial_r$ yields the Minkowski formulae
\begin{equation}\label{eq:Minkowski-formula}
  \int_\Sigma \lambda' H_{k-1} dA
  =\int_\Sigma uH_k dA,
  \qquad 1\leq k\leq n,
\end{equation} where $\lambda'(r)=\cosh r$ and $u\coloneqq \bar g(\lambda\partial_r,\nu)$ is the support function; see, e.g., \cite[Proposition 2.5]{GL-2015-IMRN}.

Using quermassintegral preserving \textit{non-local} type curvature flows, Wang--Xia proved the following sharp Alexandrov--Fenchel inequalities in \cite[Theorem 1.1]{WX1}. Here \emph{horo-convexity} means that all principal curvatures of the hypersurface are at least one.

\medskip 
\noindent{\bf Theorem} (\cite{WX1})
{\it Let $\Omega\subset\mathbb H^{n+1}$ be a bounded horo-convex domain with smooth boundary. Then for $0\leq l<k\leq n$, there holds 
\begin{equation}\label{eq000-Wk}
  W_k(\Omega)\geq W_k(B_r),
\end{equation}
where $B_r$ is a geodesic ball such that $W_l(B_r)=W_l(\Omega)$. Moreover, equality \eqref{eq000-Wk} holds if and only if $\Omega$ is a geodesic ball.}

 \medskip

  The Minkowski formulae \eqref{eq:Minkowski-formula} also underlie a quermassintegral preserving \textit{local} type curvature flow introduced by Brendle--Guan--Li in an unpublished manuscript \cite{BGL}; see
\cite[Theorem B]{HLW} and \cite[Section 4]{GL5}:
\begin{equation*}%\label{BGL-1}
  \partial_tx=\left(\frac{\lambda'}{F}-u\right)\nu,
  \qquad F=\frac{H_k}{H_{k-1}}.
\end{equation*}
They proved long-time existence and smooth convergence for strictly convex initial hypersurface when $k=n$, and for star-shaped $k$-convex initial hypersurface satisfying an additional gradient condition when $1\leq k<n$. In particular, their result gives \eqref{eq000-Wk} for $k=n$ and $0\leq l<n$ under strict convexity. Hu--Li--Wei \cite{HLW} established smooth convergence of this flow for a horo-convex initial hypersurface for every $1\leq k\leq n$, providing another proof of the Alexandrov--Fenchel inequalities \eqref{eq000-Wk}. Related sharp quermassintegral inequalities have been proved for star-shaped two-convex hypersurfaces in \cite{LWX} and for hypersurfaces with nonnegative intrinsic sectional curvature in \cite{HL1}. Locally constrained inverse curvature flows in warped product
spaces with applications to geometric inequalities were
developed by Scheuer--Xia \cite{SX}. The weighted quermassintegral inequalities for static convex domains were established in \cite{HL}.  Andrews--Chen--Wei \cite{ACW} proved quermassintegral inequalities with prescribed volume under positive sectional curvature, as well as inequalities for modified quermassintegrals under strict horo-convexity. Using duality with de Sitter space, Hu--Li \cite{HuLiDuality}
proved sharp comparisons of $W_{n-1}$ with $W_l$ when $0\leq l<n-1$
for smooth bounded strictly convex domains in hyperbolic space. For more developments in hypersurface curvature flows and geometric inequalities, we refer to \cite{AHL, Andrews-Wei2018, CGLS, CS, DeG, GeWangWu, GL-2015-IMRN, GL5, GL4, luo-wei-zhou, Makowski-Scheuer, Wei2025, WeiXiong, Wu}
and the references therein.

We next turn to hypersurfaces with capillary boundary, which meet a supporting hypersurface at a constant angle $\theta\in(0,\pi)$. The case $\theta=\pi/2$ is called a free boundary hypersurface. In the Euclidean unit ball,  \cite{WX3} and \cite{WW} studied locally constrained mean curvature flows for star-shaped hypersurfaces with free and capillary boundary, respectively. The latter work assumes $|\cos\theta|<(3n+1)/(5n-1)$. Lambert--Scheuer \cite{LambertScheuer2016, LambertScheuer2017} employed inverse mean curvature flow with free boundary to prove that strictly convex hypersurfaces converge to a flat disk, and to establish a Willmore-type inequality in dimensions $n\geq3$. Scheuer--Wang--Xia \cite{SWX} initiated quermassintegrals for
free boundary hypersurfaces and proved sharp inequalities
comparing the $n$-th quermassintegral with each lower-order one for convex hypersurfaces. Their capillary counterpart was introduced in \cite{WengX}, where these inequalities were extended to $\theta\in(0,\pi/2]$.  
Further results in this angle range include the inequalities between the quermassintegrals and the enclosed volume obtained by Hu--Wei--Yang--Zhou \cite{HWYZ} for convex hypersurfaces, and the inverse curvature flows studied by Pan--Yang \cite{PanYang} for strictly convex capillary hypersurfaces. The latter work also proved the
Alexandrov--Fenchel type inequalities for weakly convex free boundary hypersurfaces. Very recently, Pan--Scheuer \cite{PanScheuer} showed the full family of quermassintegral inequalities for $\theta\in(0,\pi/2]$ within a class of $\theta$-horocap-convex hypersurfaces.

In the Euclidean half-space, capillary quermassintegrals were introduced in \cite{WWX1} for capillary hypersurfaces, where quermassintegral inequalities comparing the $n$-th quermassintegral with those of lower orders were established for convex hypersurfaces with $\theta\in(0,\pi/2]$. A boundary tensor maximum principle was developed in \cite{HWYZhalf} and applied to prove convexity preservation for the locally constrained inverse curvature flow proposed in \cite{WWX1}, thus yielding the full family of capillary quermassintegral inequalities for the same acute angle range. By a spectral method, a general Alexandrov--Fenchel inequality and its equality characterization for mixed volumes of capillary convex bodies were obtained in \cite{MWWX}, which, in particular, yielded the full family of capillary quermassintegral inequalities for all $\theta\in(0,\pi)$. There are other interesting results concerning curvature flows and geometric inequalities for capillary hypersurfaces in the Euclidean half-space; see, for instance, \cite{DGL, Gao-Li2026, JWXZ, MWW2024, MW2025-agag, SW, WWX2} among many others.

In hyperbolic space, a locally constrained mean curvature type flow for star-shaped hypersurfaces with free boundary in a geodesic ball was studied in \cite{QWX}. The corresponding star-shaped capillary hypersurfaces supported in a geodesic ball in
space forms were investigated in \cite{MW1} under $|\cos\theta|<(3n+1)/(5n-1)$. Chai \cite{ChaiIMCF} studied the inverse mean curvature flow for
strictly convex hypersurfaces with free boundary in a geodesic ball and used its long-time convergence to a totally geodesic disk to obtain a Willmore-type inequality.  For capillary hypersurfaces supported on a totally geodesic
hyperplane, Chai--Chen \cite{CC} obtained long-time existence
and smooth convergence under star-shapedness, a containment
condition and a corresponding restriction on the contact angle.
For hypersurfaces supported on a horosphere, the first author
\cite{Guo1} proved the analogous convergence result under
star-shapedness and $|\cos\theta|<(3n+1)/(5n-1)$. Capillary quermassintegrals for capillary hypersurfaces supported on a totally geodesic hyperplane were considered in \cite{MW}, where quermassintegral inequalities comparing the $n$-th quermassintegral with those of lower orders were obtained for convex hypersurfaces with $\theta\in(0,\pi/2]$ satisfying an additional containment condition.

In this paper, we study capillary quermassintegral inequalities of capillary hypersurfaces supported on a horosphere in hyperbolic space and construct a locally constrained inverse curvature flow adapted to the horospherical boundary. We use the upper half-space model
\begin{eqnarray*} 
  \hh^{n+1}
  = \left\{x=(x_1,\ldots,x_{n+1})\in\rr^{n+1}:x_{n+1}>0 \right\},
  \qquad \bar g=\frac{1}{x_{n+1}^{2}}\delta,
\end{eqnarray*}
where $\delta$ is the Euclidean metric. Up to a hyperbolic
isometry, the horoball and its supporting horosphere are
\begin{equation*}
  \mathcal H_+=\{x\in\hh^{n+1}:x_{n+1}>1\}
\end{equation*}
and
\begin{equation*}
  \mathcal H=\partial\mathcal H_+
  =\{x\in\hh^{n+1}:x_{n+1}=1\}.
\end{equation*}
With respect to the outward unit normal
$\bar N=-E_{n+1}=(0,\ldots,0,-1)$, the horosphere $\mathcal H$
has all principal curvatures equal to one, and its induced metric
is a Euclidean metric.

From now on, we assume that hypersurface $\Sigma$ in
$\overline{\mathcal H_+}$ is smooth, compact, connected and
properly embedded, with interior in $\mathcal H_+$ and
nonempty boundary on $\mathcal H$.
We write $\widehat\Sigma$ for the bounded domain enclosed by
$\Sigma$ and its wetting domain
$\widehat{\partial\Sigma}\subset\mathcal H$, and choose the
unit normal $\nu$ along $\Sigma$ to point out of
$\widehat\Sigma$. 

A complete umbilical hypersurface in a space form has constant principal curvature $\kappa\in \mathbb{R}$. After choosing the orientation appropriately, one may assume $\kappa\geq0$. In Euclidean space, the complete umbilical hypersurfaces are only hyperplanes and round spheres, while in spherical space they are geodesic spheres, including totally geodesic great spheres. In hyperbolic space, however, they comprise four classes: geodesic spheres $(\kappa>1)$, horospheres $(\kappa=1)$, equidistant hypersurfaces $(0<\kappa<1)$, and totally geodesic hyperplanes $(\kappa=0)$. Unlike geodesic spheres, the other three cases are non-compact; see, e.g., \cite[Chapter 10]{RAF3}. Thus, capillary hypersurfaces supported on a horosphere exhibit features not present in either the geodesic-ball or the totally geodesic half-space setting.

\begin{defn} 
A hypersurface $\Sigma$ has $\theta$-capillary boundary in
$\mathcal H_+$ for constant angle $\theta\in(0,\pi)$, if
\begin{equation}\label{angle}
  \bar g(\bar N,\nu)=-\cos\theta
  \qquad\text{ on  }\partial\Sigma.
\end{equation}
The case $\theta=\pi/2$ is called a free boundary hypersurface.
\end{defn}

 The model hypersurfaces are the umbilical caps of Euclidean
radius $r>0$ given by
\begin{equation*}
  C_{\theta,r}(E_{n+1})
  \coloneqq \{x\in\mathbb R^{n+1}_+:
  |x+(r\cos\theta-1)E_{n+1}|=r,\ x_{n+1}\geq1\},
\end{equation*}
which enclose the domains
\begin{equation*}
  \widehat{C_{\theta,r}(E_{n+1})}
  \coloneqq \{x\in\mathbb R^{n+1}_+:
  |x+(r\cos\theta-1)E_{n+1}|\leq r,\ x_{n+1}\geq1\}.
\end{equation*}
Their principal curvature is $\kappa=\frac{1}{r}-\cos\theta$.
In particular, they are strictly convex precisely when
$r\cdot\cos\theta<1$. On each model cap one has
\begin{equation}\label{dfec}
  \frac{1}{x_{n+1}}-\cos\theta\bar g(x,\nu)
  =\kappa~\bar g(x-E_{n+1},\nu)
  \qquad\text{on }C_{\theta,r}(E_{n+1}),
\end{equation}
see, e.g., \cite{Guo1}. For the flow considered below (cf. Eq. \eqref{eqn:intro-equationflow}), we introduce
the conformal Killing field $X$ and the function $V$ by
\begin{equation}\label{eqn:X-V}
  X\coloneqq x-E_{n+1},\qquad V\coloneqq \frac{1}{x_{n+1}}.
\end{equation}

For $\theta\in(0,\pi/2]$, we define model functions as
\begin{equation*}
  f_k(r)\coloneqq
  W_{k,\theta}\left(\widehat{C_{\theta,r}(E_{n+1})}\right),
  \qquad 1\leq k\leq n,\quad 0<r<\frac{1}{\cos\theta},
\end{equation*}
where the capillary quermassintegrals $W_{k,\theta}$ are defined in Section \ref{sec2.4:quermassintergal}, and
$1/\cos\theta=+\infty$ when $\theta=\pi/2$.
The functions $f_k$ are strictly increasing in this interval and 
their inverses are taken in their respective images.
Our first result compares these quermassintegrals for
horo-convex hypersurfaces.

\begin{theorem}\label{mainthm-horo}
For $\theta\in(0,\pi/2]$. Let $\Sigma$ be a compact
horo-convex hypersurface with $\theta$-capillary boundary in a horoball
$\mathcal H_+$. Assume that
$E_{n+1}\in\operatorname{int}_{\mathcal H}
(\widehat{\partial\Sigma})$. Then for $1\leq l<k\leq n$,
\begin{equation}\label{AF-horo}
  W_{k,\theta}(\widehat\Sigma)
  \geq(f_k\circ f_l^{-1})
  \bigl(W_{l,\theta}(\widehat\Sigma)\bigr).
\end{equation}
Equality holds if and only if $\Sigma$ is an umbilical cap
with $\theta$-capillary boundary.
\end{theorem}

In order to establish \eqref{AF-horo}, we construct a locally constrained curvature flow for $2\leq k\leq n$. Let $x(\cdot,t):\Sigma\to\overline{\mathcal H_+}$ be a family of evolving embeddings with $x(\partial\Sigma,t)\subset\mathcal H$, satisfying
\begin{equation}\label{eqn:intro-equationflow}
\begin{cases}
  (\partial_t x)^\perp
  =\left(\dfrac{V-\cos\theta\bar g(x,\nu)}{H_k/H_{k-1}}
  -\bar g(X,\nu)\right)\nu,
  &\text{in }\Sigma\times[0,T),\\
  \bar g(\nu,\bar N\circ x)=-\cos\theta,
  &\text{on }\partial\Sigma\times[0,T),\\
  x(\cdot,0)=x_0,
  &\text{on }\Sigma.
\end{cases}
\end{equation}
This flow preserves $W_{k,\theta}$ and makes $W_{l,\theta}$ nondecreasing for $1\leq l<k$; see Proposition \ref{Wk-Wn}. For horo-convex initial hypersurfaces, we obtain long-time existence and smooth convergence to an umbilical cap. The resulting comparison with the limit gives \eqref{AF-horo}.

For $k=n$, the flow \eqref{eqn:intro-equationflow} also derives inequalities under convexity and a containment condition. We include the Gauss--Bonnet--Chern type theorem for the highest-order quermassintegral $W_{n+1,\theta}$ in the following statement.

\begin{theorem}\label{mainthm}
For $\theta\in(0,\pi/2]$. Let $\Sigma$ be a compact convex hypersurface with $\theta$-capillary boundary in a horoball $\mathcal H_+$. Assume that $E_{n+1}\in\operatorname{int}_{\mathcal H}(\widehat{\partial\Sigma})$ and there is a constant $R>0$ such that
\begin{equation}\label{assump}
  \widehat\Sigma\subseteq
  \widehat{C_{\theta,R}(E_{n+1})},
  \qquad R\cdot\cos\theta<1.
\end{equation}
Then for $1\leq l\leq n-1$,
\begin{equation*}
  W_{n,\theta}(\widehat\Sigma)
  \geq(f_n\circ f_l^{-1})
  \bigl(W_{l,\theta}(\widehat\Sigma)\bigr).
\end{equation*}
Equality holds if and only if $\Sigma$ is an umbilical cap with $\theta$-capillary boundary.

Moreover, the highest-order quermassintegral given by \eqref{Wn-1113} satisfies  \begin{equation}\label{eq:top-order-GB}
  W_{n+1,\theta}(\widehat\Sigma)
  =\frac{\omega_{n-1}}{n+1}
  \int_0^\theta\sin^{n-1}s~ds,
\end{equation}where $\omega_{n-1}=|\mathbb S^{n-1}|$. Thus, for fixed $n$ and $\theta$, this quantity is independent of $\Sigma$ within the stated convex class. \end{theorem}
\begin{remark}
\

\begin{itemize}
\item [(a)]The normalization $E_{n+1}\in\operatorname{int}_{\mathcal H} (\widehat{\partial\Sigma})$ can always be achieved by a horizontal translation.
\item [(b)] The containment condition \eqref{assump} ensures that $V-\cos\theta\bar g(x,\nu)>0$ by Proposition \ref{bounded}, and holds automatically for horo-convex hypersurfaces by Lemma \ref{lemma-horo-umbilical-cap}.
\item [(c)]For $\theta\in(0,\pi/2]$ and convex hypersurface $\Sigma$, the boundary $\partial\Sigma$ is convex in both $\mathcal H$ and $\Sigma$ by Corollary \ref{cor222}. This boundary convexity enters the mean curvature estimate \eqref{ffff}.
\end{itemize}
\end{remark}

In the free boundary case $\theta=\pi/2$, the containment
condition \eqref{assump} holds for all sufficiently large $R$. Theorem \ref{mainthm} therefore gives the following result.

\begin{corollary}
Let $\Sigma$ be a compact convex hypersurface with free boundary in $\mathcal H_+$. Assume that $E_{n+1}\in\operatorname{int}_{\mathcal H}(\widehat{\partial\Sigma})$. Then for $1\leq l\leq n-1$,
\begin{equation*}
  W_{n,\pi/2}(\widehat\Sigma)
  \geq(f_n\circ f_l^{-1})
  \bigl(W_{l,\pi/2}(\widehat\Sigma)\bigr),
\end{equation*}
where the model functions are taken with $\theta=\pi/2$. Equality holds if and only if $\Sigma$ is a free boundary umbilical cap. \end{corollary}

When $n=2$, expanding the quermassintegrals in Theorem
\ref{mainthm} gives the following Minkowski type inequality
and the Gauss--Bonnet type formula for the total Gauss curvature.
 
\begin{corollary}
Let $\Sigma$ be a surface in $\mathcal{H}_{+}$ that satisfies the assumptions of Theorem \ref{mainthm}. Then
\begin{eqnarray*}
 2\int_\Sigma H_1dA
\geq 6(f_2\circ f_1^{-1})
  \left(\frac13\bigl(|\Sigma|
  -\cos\theta|\widehat{\partial\Sigma}|\bigr)\right)+2|\widehat\Sigma|-(1+\cos^2\theta)|\widehat{\partial\Sigma}|+\sin\theta\cos\theta|\partial\Sigma|.
\end{eqnarray*}
Equality holds if and only if $\Sigma$ is an umbilical cap
with $\theta$-capillary boundary. Moreover, there holds
\begin{equation*}
  3W_{3,\theta}(\widehat\Sigma)
  =\int_\Sigma H_2dA-|\Sigma|+\sin\theta|\partial\Sigma|
  =2\pi(1-\cos\theta).
\end{equation*}

\end{corollary}

\medskip 

\noindent\textit{The paper is organized as follows.} In Section \ref{sec2}, we collect the geometric identities and first variation formulae. In Section \ref{sec4}, we construct the locally constrained curvature flow and derive its scalar formulation, basic estimates and evolution equations. Section \ref{curvature-sec} establishes curvature bounds, long-time existence and smooth convergence. Finally, Section \ref{sec5} proves the Alexandrov--Fenchel inequalities for horo-convex and convex hypersurfaces stated in Theorems \ref{mainthm-horo} and \ref{mainthm}, respectively.

\
 
\section{Preliminaries}\label{sec2}

Let $x:(\Sigma,g)\to(\overline{\mathcal H_+},\bar g)$ be an isometric immersion of a compact orientable $n$-manifold with boundary satisfying $x(\operatorname{int}\Sigma)\subset\mathcal H_+$ and $x(\partial\Sigma)\subset\mathcal H$.
This immersion is said to be supported on $\mathcal H$.

We denote by $\bar \n$, $\bar \Delta$ and $\bar \n^2$ the gradient, the Laplacian and the Hessian on $\mathcal{H}_{+}$ w.r.t. $\bar g$ respectively, while by $\n$, $\Delta$ and $\n^2$ the gradient, the Laplacian and the Hessian on $\Sigma$ w.r.t. its induced metric $g$ respectively. We retain the outward unit normal $\nu$ as above. We denote by $\bar N$ the unit outward normal to $\mathcal{H}$ in $\mathcal{H}_{+}$ and $\mu$ be the unit outward normal to $\p \Sigma$ in $\Sigma$.
Let $\bar \nu $ be the unit normal to $\p \Sigma$ in $\mathcal{H}$ such that the bases $\{\nu, \mu\}$ and $\{\bar \nu  , \bar N\}$ have the same orientation in the
normal bundle of $\p \Sigma\subset \mathcal{H}_{+}$. See Figure \ref{figure1} below. To be more precise, in the normal bundle of $\p \Sigma$, the following relations hold:  
 \begin{eqnarray}
&&\mu=\sin \th ~ \bar N+\cos \th~ \bar \nu ,  \label{mu0}
\\&&\nu=-\cos \th ~ \bar N+\sin \th~ \bar \nu  .\label{nu0}
\end{eqnarray}
\begin{figure}[H]
	\centering
\includegraphics[height=7cm,width=12cm]{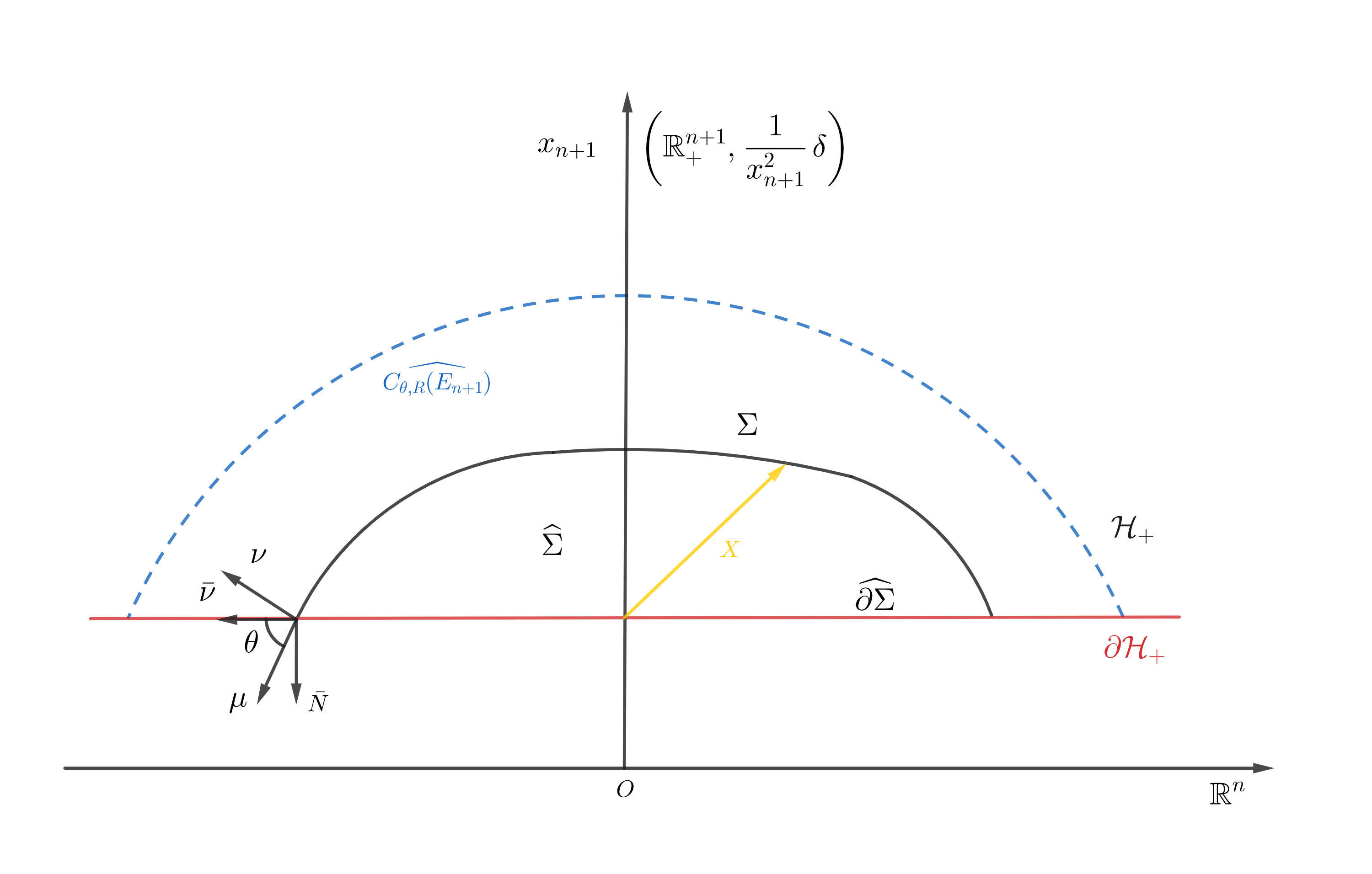}
	\caption{Hypersurface $\Sigma$ with $\theta$-capillary boundary in a horoball $\mathcal{H}_{+}$.} \label{figure1}
\end{figure}

Denote $x$ by the position vector in $\hh^{n+1}$ and $\bar{\nabla}$  the Levi-Civita connection of $\hh^{n+1}$. Let $\{E_i\}_{i=1}^{n+1}$ be the canonical basis of $\rr^{n+1}$.
We use $\<\cdot ,\cdot\>$ and $\bar g$ to denote the inner product of $\rr^{n+1}$ and $\hh^{n+1}$ respectively,  $D$ and $\bar \nabla$ to denote the Levi-Civita connection of $\rr^{n+1}$ and $\hh^{n+1}$ respectively.  Let $\bar E_i = x_{n+1} E_i$. Then $\{ \bar E_i\}_{i=1}^{n+1}$  is an orthonormal basis of $\hh^{n+1}$.
The relationship of $\bar{\nabla}$ and $D$ is given by
\begin{equation*}
  \bar{\nabla}_{Y}Z=D_{Y}Z-Y(\ln x_{n+1})Z-Z(\ln x_{n+1})Y+\langle Y,Z\rangle D(\ln x_{n+1}).
\end{equation*}
It follows that
\begin{eqnarray*}  
\bar \n_Y x &=& -\bar g (Y, \bar E_{n+1})x + \bar g (Y, x ) \bar E_{n+1}, \\ 
\bar \n_Y  E_i &=& -\bar g (Y, \bar E_{n+1}) E_i + \bar g (Y, \bar E_i) E_{n+1}, \quad \forall \,i=1,2, \cdots, n, \\
\bar \n_Y E_{n+1}  &=& - \frac 1{x_{n+1}} Y,
\end{eqnarray*}
for any vector $Y$ in $\hh^{n+1}$, which will be used later.

The following two facts play an important role in showing the Minkowski type formula in $\mathcal{H}_{+}$.
\begin{prop}[{\rm \cite[Proposition 2.3]{GWX}}]\label{lem2.1}
\

$(\rm i)$  The vector field  ~$x$ is a Killing vector field in $\hh^{n+1}$, i.e,
\begin{eqnarray}\label{xKilling}
\frac12\left( \bar g (\bar \n_i x, E_j) + \bar g (\bar \n_j x, E_i) \right ) =0.
\end{eqnarray}

$(\rm {ii})$~$E_{n+1}$ is a conformal Killing vector field in $\hh^{n+1}$, i.e,
\begin{eqnarray*}
\frac12\left (\bar g( \bar \n_i E_{n+1}, E_j)+
\bar g( \bar \n_j E_{n+1}, E_i)
\right)= -\frac{1}{x_{n+1}}\bar{g}_{ij}.
\end{eqnarray*}
Here $\bar \n_i = \bar \n _{E_i}$ and $\bar g_{ij}= \bar g(E_i, E_j)$.
\end{prop}

For $1 \leq k \leq n$. Let $\sigma_{k}$ be the $k$-th elementary symmetric polynomial evaluated at the principal curvatures of $\Sigma$ and $H_{k}=\binom{n}{k}^{-1} \sigma_{k}$ for $1 \leq k \leq n$. We make the convention that $H_{0}=1$ and $H_{n+1}=0$.

We denote $\sigma_k(\kappa)$ with $\kappa_i$ omitted by $\sigma_k(\kappa|i)$, so that $\partial\sigma_{k+1}/\partial\kappa_i=\sigma_k(\kappa|i)$.
Set $\sigma_k^{ij}=\partial\sigma_k/\partial h_{ij}$,
the $(k-1)$-st Newton tensor. We recall the following basic identities.
\begin{prop}\label{formula2}For all $1 \leq i \leq n$, then
\begin{itemize}
 \item [(1)]$\sigma_{k}^{ij}$ is divergence-free, i.e., $\nabla_{i}\sigma_{k}^{ij}=0$;
\item [(2)]$\sigma_{k}(\kappa)=\sigma_{k}(\kappa |i)+\kappa_{i} \sigma_{k-1}(\kappa |i)$;
\item [(3)]$\sum_{i=1}^n\sigma_k(\kappa|i)=(n-k)\sigma_k(\kappa)$;
\item [(4)] $\sum_{i=1}^{n}\kappa_{i}\sigma_{k-1}(\kappa|i)=k\sigma_{k}(\kappa)$;
\item [(5)]$\sum_{i=1}^{n} \kappa_{i}^{2}\sigma_{k-1}(\kappa|i)=\sigma_{1}(\kappa)\sigma_{k}(\kappa)-(k+1) \sigma_{k+1}(\kappa)$.
\end{itemize}
\end{prop}
Denote $$\Gamma_{+}\coloneqq \left\{\kappa \in \mathbb{R}^{n}: \kappa_{i}>0,1 \leq i \leq n\right\}.$$ % be a G{\aa}rding cone.
\begin{prop}\label{Hk} For $1 \leq l<k \leq n$ and $\kappa \in \Gamma_{+}$, then
\begin{equation}\label{Newton}
  H_{l} H_{k-1} \geq H_{l-1} H_{k},
\end{equation}
with equality if and only if $\kappa=c(1,\ldots,1)$ for some $c>0$. 
\end{prop}
\begin{prop}\label{garding-cone}
For $1\leq k\leq n$, the function
$F=H_k/H_{k-1}$ is smooth, positive, strictly increasing in each
principal curvature, and concave on $\Gamma_+$.
\end{prop}
The algebraic identities in Proposition \ref{formula2} follow
from \cite[Lemma 2.14]{Spruck}; the divergence-free property of
the Newton tensor follows from the Codazzi equation, as in
\cite[Corollary 1.7]{Spruck}.
Propositions \ref{Hk} and \ref{garding-cone} follow from
\cite[Proposition 2.7 and Theorems 2.11 and 2.16]{Spruck}.

\begin{prop}\label{lem:shifted-inverse}
If $\kappa_{i}>1$ for every $1 \leq i \leq n$ and $B\coloneqq (B_{ij})$ is a real symmetric matrix, then
\begin{equation*}%\label{eq:inverse-matrix}
 F^{ij,kl}B_{ij}B_{kl}
 +2\sum_i\sum_{\kappa_p>1}\frac{\dot{F}^i}{\kappa_p-1}B_{ip}^2
 \ge\frac{2}{F-1}\left(\sum_i\dot{F}^iB_{ii}\right)^2,
\end{equation*}
where $\dot{F}^i(\kappa)\coloneqq \partial F/\partial\kappa_i$ and $F^{ij,kl}$ denotes the
second derivative of the spectral matrix function at a diagonal matrix.
\end{prop}

\begin{proof}
Due to $\kappa_{i}>1$, define %we can define a positive and inverse-concave function by
\begin{equation*}%\label{eq:G-a}
 G(\lambda)\coloneqq F(1+\lambda_1,\ldots,1+\lambda_n)-1,
 \qquad \lambda\in \Gamma_+,
\end{equation*} which is a positive and inverse-concave function. From \cite[lemma 2.2]{WWZ} and \cite[Section 2]{AndrewsMcCoyZheng} for $G$, the assertion follows.

\end{proof}

\subsection{Basic properties for capillary hypersurfaces in horoball}\

Let $\Sigma$ be a smooth, properly embedded hypersurface with $\theta$-capillary boundary, parametrized by $x:\overline{\mathbb B}^{n}\to\overline{\mathcal H_+}$. We keep the normal vector field as above and the relations \eqref{mu0}--\eqref{nu0} to compare the boundary curvatures. The boundary $\partial \Sigma$ can be viewed as a smooth closed hypersurface in the horosphere $\mathcal{H}$, which bounds $\widehat{\partial \Sigma}$ inside $\mathcal{H}$.

The second fundamental form of $\partial \Sigma$ in $\mathcal{H}$ is given by
\begin{equation*}
\widehat{h}(X, Y)\coloneqq -\bar{g}(\nabla_{X}^{\mathcal{H}} Y, \bar{\nu})=-\bar{g}(\bar{\nabla}_{X} Y, \bar{\nu}), \quad X, Y \in T(\partial \Sigma).
\end{equation*}
The second equality holds since $\bar{g}(\bar{\nu}, \bar N \circ x)=0$. 

The second fundamental form of $\partial \Sigma$ in $\Sigma$ is given by
\begin{equation*}
\tilde{h}(X, Y)\coloneqq -\bar{g}(\nabla_{X} Y, \mu)=-\bar{g}(\bar{\nabla}_{X} Y, \mu), \quad X, Y \in T(\partial \Sigma).
\end{equation*}
The second equality holds since $\bar{g}(\nu, \mu)=0$.
\begin{prop}\label{formula11}
Let $\Sigma\subseteq\mathcal{H}_{+}$ be a hypersurface with $\theta$-capillary boundary and $\{e_{\alpha}\}_{\alpha=1}^{n-1}$ be
an orthonormal frame of $\partial\Sigma$. Then along $\partial\Sigma$,
\begin{itemize}
  \item [(1)]$\mu$ is a principal direction of $\Sigma$, that is, $h\left(\mu, e_{\alpha}\right)=0$;
 \item [(2)]$h_{\alpha \beta}=\sin \theta \widehat{h}_{\alpha \beta}-\cos \theta \delta_{\alpha \beta}$;
 \item [(3)]$\tilde{h}_{\alpha \beta}=\cos \theta \widehat{h}_{\alpha \beta}+\sin \theta \delta_{\alpha \beta}=\frac{1}{\sin \theta} \delta_{\alpha \beta}+\cot \theta h_{\alpha \beta}$;
\item [(4)]$h_{\alpha \beta ; \mu}=\tilde{h}_{\beta \gamma}\left(h_{\mu \mu} \delta_{\alpha \gamma}-h_{\alpha \gamma}\right)$.
\end{itemize}
\end{prop}
\begin{proof}
The first assertion (1) follows from \cite[Proposition 2.1]{WX2}. For (2), it follows from
\begin{equation*}
  h_{\alpha \beta}=-\bar{g}(\bar{\nabla}_{e_{\alpha}} e_{\beta}, \nu)=\bar{g}(\widehat{h}_{\alpha \beta} \bar{\nu}+\delta_{\alpha \beta} \bar{N}, \nu)=\sin \theta \widehat{h}_{\alpha \beta}-\cos \theta \delta_{\alpha \beta}.
\end{equation*}
Similarly for (3), there holds
\begin{equation*}
  \tilde{h}_{\alpha \beta}=-\bar{g}( \bar{\nabla}_{e_{\alpha}} e_{\beta}, \mu)=\bar{g}(\widehat{h}_{\alpha \beta} \bar{\nu}+\delta_{\alpha \beta} \bar{N}, \mu)=\cos \theta \widehat{h}_{\alpha \beta}+\sin \theta \delta_{\alpha \beta}.
\end{equation*}
To show the last assertion (4), by taking the derivative of $h\left(\mu, e_{\alpha}\right)=0$ with respect to $e_{\beta}$, using the Codazzi equation in $\mathbb{H}^{n+1}$ and the fact (1), we get
\begin{equation*}
\begin{aligned}
0=e_{\beta}\left(h\left(\mu, e_{\alpha}\right)\right) &=h_{\alpha \mu ; \beta}+h\left(\nabla_{e_{\beta}} e_{\alpha}, \mu\right)+h\left(\nabla_{e_{\beta}} \mu, e_{\alpha}\right) \\
&=h_{\alpha \beta ; \mu}+\bar{g}(\nabla_{e_{\beta}} e_{\alpha}, \mu) h_{\mu \mu}+\bar{g}(\nabla_{e_{\beta}} \mu, e_{\gamma}) h_{\alpha \gamma} \\
&=h_{\alpha \beta ; \mu}-\tilde{h}_{\beta \gamma}\left(h_{\mu \mu} \delta_{\alpha \gamma}-h_{\alpha \gamma}\right).
\end{aligned}
\end{equation*}
Hence we complete the proof. 
\end{proof}
As a direct consequence of Proposition \ref{formula11}, we have the following result.
\begin{corollary}\label{cor222}
Let $\Sigma$ be a hypersurface with $\theta$-capillary boundary for $\theta \in\left(0, \pi/2 \right]$ in $\mathcal{H}_{+}$. If $\Sigma$ is convex (in the sense $h$ is nonnegative definite), then $\partial \Sigma$ is convex in both $\mathcal{H}$ and $\Sigma$.
\end{corollary}

For the flow \eqref{eqn:intro-equationflow}, positivity of the
numerator ensures parabolicity on $\Gamma_+$, while positivity
of the support function allows a radial graph formulation. Denote 
\begin{equation*}
 A_\theta\coloneqq V-\cos\theta\bar g(x,\nu).
\end{equation*}
The following estimates are obtained both from comparison with two
strictly convex umbilical caps.
\begin{prop}\label{bounded} For $\theta \in\left(0, \pi/2 \right]$. Let $\Sigma$ be a convex hypersurface with $\theta$-capillary boundary in $\mathcal{H}_{+}$. Assume that $E_{n+1}\in\operatorname{int}_{\mathcal H}(\widehat{\partial\Sigma}) $ and there exist $0<R_{1}<R_{2}<\infty$ such that
\begin{equation*}
\widehat{C_{\theta, R_{1}}(E_{n+1})}\subseteq \widehat{\Sigma} \subseteq \widehat{C_{\theta, R_{2}}(E_{n+1})}\quad \text{and}~~~\ R_{2}  \cos\theta<1.
\end{equation*}
Then  
\begin{itemize}
  \item [{\rm(i)}]$1\leq x_{n+1}\leq R_{3}\coloneqq 1+R_{2}(1-\cos\theta)$;
  \item [{\rm(ii)}]$A_{\theta}\geq\frac{(1-\cos\theta)(1-R_{2}\cos\theta)}{R_{3}}$;
\item [{\rm(iii)}]$\bar{g}\left(E_{n+1}, \nu\right)\geq\frac{\cos\theta}{R_{3}}$;
  \item [{\rm(iv)}]$\bar{g}(X,\nu)\geq\frac{(1-\cos\theta)R_{1}}{R_{3}}$.
\end{itemize}
\end{prop}
\begin{proof}
Due to $\Sigma\subseteq\widehat{C_{\theta, R_{2}}(E_{n+1})}\subseteq\mathcal{H}_{+}$, thus we have
\begin{equation*}
  1\leq x_{n+1}\leq R_{2}+(1-R_{2}\cos\theta)=1+R_{2}(1-\cos\theta),
\end{equation*}
hence (i) holds. By $\Sigma\subseteq\widehat{C_{\theta, R_{2}}(E_{n+1})}$ again, we imply
\begin{eqnarray*}
\langle x,\nu_{\delta}\rangle&=&\langle x-(1-R_{2}\cos\theta)E_{n+1},\nu_{\delta}\rangle+(1-R_{2}\cos\theta)\langle E_{n+1},\nu_{\delta}\rangle\\
&\leq&|x-(1-R_{2}\cos\theta)E_{n+1}|+(1-R_{2}\cos\theta)\nonumber\\
&\leq&R_{2}+(1-R_{2}\cos\theta).\nonumber
\end{eqnarray*}
Here $\nu_{\delta}$ is the unit normal vector on $\Sigma$ with respect to the Euclidean metric $\delta$.
It implies that
\begin{equation*}
  A_{\theta}=\frac{1-\cos\theta\langle x,\nu_{\delta}\rangle}{x_{n+1}}\geq\frac{1-\cos\theta(R_{2}+(1-\cos\theta R_{2}))}{R_{3}}
=\frac{(1-\cos\theta)(1-\cos\theta R_{2})}{R_{3}},
\end{equation*}
thus (ii) holds.

Since $\Sigma$ and $\mathcal{H}$ are both convex hypersurfaces in $\mathbb{H}^{n+1}$, which implies their enclosed region $\widehat{\Sigma}$ is a bounded convex body in $\mathbb{H}^{n+1}$. From Corollary \ref{cor222}, we know that the wetting domain $\widehat{\partial\Sigma}=\overline{\widehat{\Sigma}}\cap\mathcal{H}$ is a Euclidean convex body in $\mathcal{H}\simeq \mathbb{R}^{n}$.

We first claim that the vertical projection of $\widehat\Sigma$ is contained in
$\widehat{\partial\Sigma}$. In fact, the totally geodesic plane through $p=(p',1)\in{\partial\Sigma}$ and tangent to $\Sigma$ is
\begin{equation*}%\label{totally-geodesic-plane}
\Pi\coloneqq \begin{cases}\{(x',x_{n+1}):|x'-p'+\tan\theta\bar{\nu}_{\delta}(p')|^{2}+x^{2}_{n+1}=\sec^{2}\theta\}, & \text { when } \theta\in(0,\pi/2) \\
\{(x',x_{n+1}): \langle x'-p',\bar{\nu}_{\delta}(p')\rangle=0\}, & \text{ when } \theta=\pi/2,
\end{cases}
\end{equation*}
where $\bar{\nu}_{\delta}=\frac{\bar{\nu}}{x_{n+1}}$ is the unit normal vector of $\partial\Sigma$ in $\mathcal{H}$ w.r.t the Euclidean metric $\delta$. 

By the hyperbolic convexity, we know that $\widehat\Sigma$ always lies on one side of the supporting plane $\Pi$. Thus for any $(y',y_{n+1})\in\widehat{\Sigma}$, we see
\begin{equation}\label{disance}
  2\sin\theta\langle y'-p',\bar{\nu}_{\delta}(p')\rangle+\cos\theta(|y'-p'|^{2}+y^{2}_{n+1}-1)\leq0.
\end{equation}
By contradiction, if there exists a point $(y',y_{n+1})\in\widehat{\Sigma}$ such that $y'\notin\widehat{\partial\Sigma}$. We choose $(p',1)\in\partial\Sigma$ is the nearest point to $(y',1)$ and denote their distance by $d$. Thus $y'-p'=d\bar{\nu}_{\delta}(p')$ since $\widehat{\partial\Sigma}$ is a Euclidean convex body.

From \eqref{disance}, we have 
\begin{equation}\label{distant-cos}
    2d\sin\theta+\cos\theta(d^{2}+y^{2}_{n+1}-1)\leq0.
\end{equation}
However, the left-hand side of \eqref{distant-cos} is strictly positive since $\cos\theta\geq0$, which is a contradiction. Hence the claim holds.

For any $p=(p',1)\in\widehat{\partial\Sigma}$, we consider a hyperbolic geodesic line through point $p$ as follows
\begin{equation*}
  \gamma_{p}=\{(p',x_{n+1})\in\mathbb{R}^{n+1}_{+}:x_{n+1}>0\}.
\end{equation*}
Since $\widehat{\Sigma}$ is a hyperbolic convex body, thus the geodesic line segment $[\tilde{p},\tilde{q}]_{\mathbb{H}^{n+1}}\in\widehat{\Sigma}$ for any two points $\tilde{p},\tilde{q}\in\widehat{\Sigma}$, which implies that $\gamma_{p}\cap\widehat{\Sigma}$ is a connected interval. 
Due to $\theta \in\left(0, \pi/2 \right]$, we have $\Sigma$ is a vertical graph, hence $\bar{g}(E_{n+1},\nu)\geq0$ along $\Sigma$.

Let $h^{\delta}$ be the second fundamental form with respect to ${\nu}_{\delta}$ of ${\Sigma}$ in $\left(\mathbb{R}_{+}^{n+1}, \delta_{\mathbb{R}_{+}^{n+1}}\right)$.  Then $h^{\delta}$ and $h$ have the following relations
\begin{equation}\label{conformal-relation}
x_{n+1}h^{\delta}_{ij}=h_{ij}+\langle E_{n+1},{\nu}_{\delta}\rangle\delta_{ij}\geq0.
\end{equation}
Hence $\Sigma$ is a Euclidean convex hypersurface in $\left(\mathbb{R}_{+}^{n+1}, \delta_{\mathbb{R}_{+}^{n+1}}\right)$.
By Euclidean convexity and the boundary angle condition \eqref{angle}.
%$\langle E_{n+1},\nu_\delta\rangle\ge\cos\theta$.
Therefore,
\begin{equation*}
\bar g(E_{n+1},\nu)
=\frac{\langle E_{n+1},\nu_\delta\rangle}{x_{n+1}}
\ge\frac{\cos\theta}{R_3},
\end{equation*}
which proves (iii).

At $x\in\Sigma$, we choose the vector
\begin{equation*}
y=(1-\cos\theta R_{1})E_{n+1}+R_{1}\nu_{\delta}(x)\in C_{\theta, R_{1}}(E_{n+1}).
\end{equation*}
Since $\Sigma$ is a Euclidean convex hypersurface in $\left(\mathbb{R}_{+}^{n+1}, \delta_{\mathbb{R}_{+}^{n+1}}\right)$. Due to ${C_{\theta, R_{1}}(E_{n+1})}\subseteq \widehat{\Sigma}$. We utilize the supporting hyperplane inequality for the convex body $\widehat{\Sigma}$ to get that
\begin{equation*}
  \langle x-y,\nu_{\delta}\rangle\geq0,
\end{equation*}
and
\begin{equation*}
  \bar{g}(X,\nu)=\frac{\langle x-E_{n+1},\nu_{\delta}\rangle}{x_{n+1}}\geq\frac{\langle y-E_{n+1},\nu_{\delta}\rangle}{R_{3}}=\frac{(1-\cos\theta\langle E_{n+1},\nu_{\delta}\rangle)R_{1}}{R_{3}}\geq\frac{(1-\cos\theta)R_{1}}{R_{3}}.
\end{equation*}
\end{proof}

\subsection{General variation formulae} \

Before specifying the curvature speed, we record the variation formulae for a general capillary deformation. Let $\Sigma_t$, parametrized by $x(\cdot,t):\Sigma\to\overline{\mathcal H_+}$, be a smooth family of embedded hypersurfaces with $\theta$-capillary boundary, evolving by
%Let $\Sigma_t$ be a family of smoothly embedded hypersurfaces with $\theta$-capillary boundary in $\mathcal{H}_{+}$, given by $x(\cdot, t): \Sigma \rightarrow\mathcal{H}_{+}$ as the general flow
\begin{equation}\label{generalflow}
 \partial_t x=f \nu+\mathcal T,
\end{equation}
with some normal speed function $f$ and $\mathcal T \in T \Sigma_t$. Due to $x(\partial\Sigma, t)\subseteq\mathcal{H}$, thus
\begin{equation*}
  f \nu+\mathcal T|_{\partial \Sigma_{t}}=\partial_{t} x|_{\partial \Sigma_{t}} \in T \mathcal{H}.
\end{equation*}
By virtue of \eqref{mu0} and \eqref{nu0}, 
\begin{equation*}
  \nu=\frac{1}{\sin \theta} \bar{\nu}-\cot \theta \mu.
\end{equation*}
Since $\bar{\nu} \in T \mathcal{H}$, we see $(\mathcal T-f \cot \theta \mu) \in T \mathcal{H} \cap T \Sigma_{t}$. Up to a diffeomorphism of $\mathbb{S}^{n-1}$, we may always assume that
\begin{equation}\label{boundary2}
\mathcal T|_{\partial \Sigma_{t}}=f \cot \theta \mu .
\end{equation}

Denote by $h$ and $H$ the second fundamental form and the mean curvature of the immersion $x$, respectively. Precisely,
$h(X, Y)= \bar g(\bar \n_X \nu, Y)$ and $H=\tr_g(h)$. The evolution equations below have been shown in \cite{HP} for $\mathcal T=0$ and in \cite[Proposition A.1]{GLX} for a general $\mathcal T$. They will be used to derive the a priori
estimates.
\begin{prop}\label{formula111}
Along flow \eqref{generalflow}, there hold
\begin{itemize}
  \item [(1)]$\partial_{t} g_{i j}=2 f h_{i j}+\nabla_{i} \mathcal T_{j}+\nabla_{j} \mathcal T_{i}$.
  \item [(2)]$\partial_{t} d A_{t}=(f H+\operatorname{div}(\mathcal T)) d A_{t}$.
  \item [(3)]$\partial_{t} \nu=-\nabla f+h\left(e_{i},\mathcal T\right) e_{i}$.

  \item [(4)]$\partial_{t} h_{i j}=-\nabla_{i j}^{2} f+f (h_{i k} h_{j}^{k}+\bar{g}_{ij})+\nabla_{\mathcal T} h_{i j}+h_{j}^{k} \nabla_{i} \mathcal T_{k}+h_{i}^{k} \nabla_{j} \mathcal T_{k}$.

\item [(5)]$\partial_t h_j^i=-\nabla^i\nabla_jf-f(h_j^kh_k^i-\bar g_j^i)+\nabla_{\mathcal T}h_j^i+h_k^i\nabla_j\mathcal T^k-h_j^k\nabla_k\mathcal T^i$.

  \item [(6)]$\partial_{t} H=-\Delta f+(n-|h|^{2})f+\nabla_{\mathcal T} H$.

\item [(7)]$\partial_tF=-F_i^j\nabla^i\nabla_jf
-f(F_i^jh_j^kh_k^i-F_i^i)+\nabla_{\mathcal T}F$
for a smooth symmetric function $F$ of the principal curvatures,
where $F_i^j\coloneqq\partial F/\partial h_j^i$.

  \item [(8)]$\partial_{t} \sigma_{k}=-\frac{\partial \sigma_{k}}{\partial h_{i}^{j}} \nabla^{i} \nabla_{j} f-f\left(\sigma_{1} \sigma_{k}-(k+1) \sigma_{k+1}\right)+(n-k+1)f\sigma_{k-1}+\nabla_{\mathcal T} \sigma_{k}$.
\end{itemize}
\end{prop}

\begin{prop}\label{boundary-equ} Let $\Sigma_{t} \subset \mathcal{H}_{+}$ be a family of smooth, properly embedded hypersurfaces with $\theta$-capillary boundary. Then
\begin{equation*}
  \nabla_{\mu} f=\left(\frac{1}{\sin \theta}+\cot \theta h_{\mu \mu}\right) f \qquad\text { on } \partial \Sigma_{t} .
\end{equation*}
\end{prop}
\begin{proof} Let $\left\{e_{\alpha}\right\}_{\alpha=1}^{n-1}$ be an orthonormal frame of $T(\partial \Sigma_{t})\subset\mathcal{H}$. Then $\left\{\left\{e_{\alpha}\right\}_{\alpha=1}^{n-1}, \mu\right\}$ forms an orthonormal frame for $T \Sigma_{t}$. Since
\begin{equation*}
 \bar{g}(\nu, \bar{N} \circ x)=-\cos \theta\qquad \text { on } \partial \Sigma_{t},
\end{equation*}
by taking the time derivative and using Proposition \ref{formula111} (3), we obtain along $\partial \Sigma_{t}\subseteq\mathcal{H}$,
\begin{eqnarray*}
0 &=&\bar{g}(\partial_{t} \nu, \bar{N})+\bar{g}(\nu, \bar{\nabla}_{f \nu+\mathcal T} \bar{N}) \\
&=&\bar{g}(-\nabla f+h(e_{i}, \mathcal T)e_{i},\sin\theta\mu-\cos\theta\nu)+\bar{g}(\nu, f \nu+\mathcal T)\\
&=&-\sin \theta \nabla_{\mu} f+\cos \theta h(\mu,\mu) f+f.
\end{eqnarray*}
The proof is complete.
\end{proof}

\subsection{Minkowski formulae}\ 

We recall the properties of the vector field $X$ and the function $V$ introduced by \eqref{eqn:X-V}. Proposition \ref{lem2.1} gives the following identities, which will be used in the capillary Minkowski formula.

\begin{prop}[{\rm \cite[Proposition 2.3]{GX2}}]%\label{xaa}

\

$(i)$~$X$ is a conformal Killing vector field with $\frac{1}{2}\mathcal{L}_{X}\bar{g}=V\bar{g}$, namely
\begin{eqnarray}\label{XXaeq1}
\frac12\left[\bar \n_i (X)_j+ \bar \n_j (X)_i\right]=V\bar{g}_{ij}.
\end{eqnarray}

 $ ({ii})$ $X\mid_{\mathcal{H}}$ is a tangential vector field on $\mathcal{H}$, i.e.,
\begin{eqnarray}\label{XXaeq2}
\bar{g}(X, \bar{N})=0 \quad\text{\rm on}~~\mathcal{H}.
\end{eqnarray}
\end{prop}
\begin{prop}[{\rm \cite[Proposition 2.4]{GX2}}]\label{xaa2} ~ Conformal factor $V$ satisfies the following properties:
\
\begin{eqnarray}
  \bar{\n}^2 V&=& V  \bar{g } \quad\quad\text{in}~~ \hh^{n+1},\label{Va-1}\\
\partial_{\bar{N}}V&=& V \qquad  \qquad~~\text{on}~~ \mathcal{H}.\nonumber
\end{eqnarray}
\end{prop}
By employing such properties of the conformal Killing vector $X$, %Chen--Pyo \cite[Proposition 4]{CP} established the following higher-order capillary Minkowski formula in a horoball.
Guo--Wang--Xia \cite[Proposition 3.4]{GWX} established the first-order ($k=1$) capillary Minkowski formula in a horoball. The higher-order identities below follow from Chen--Pyo \cite[Proposition 4]{CP} by specializing the supporting hypersurface to a horosphere.
\begin{prop}\label{mink1}
For $\theta\in(0,\pi)$. Let $\Sigma$ be a smooth compact immersed orientable hypersurfaces with $\theta$-capillary boundary in $\mathcal{H}_{+}$. Then
\begin{equation}\label{mink2}
 \int_{\Sigma} H_{k-1}\left(V-\cos \theta \bar{g}(x,\nu)\right)d A=\int_{\Sigma} H_{k}\bar{g}(X, \nu) dA,\quad \text{for }~ 1 \leq k \leq n,
\end{equation}
where ${x}$ is position vector field and $H_{k}$ is the $k$-th normalized mean curvature of $\Sigma$.
\end{prop}

\subsection{Capillary quermassintegrals and first variations}\label{sec2.4:quermassintergal}\ 

 The boundary corrections in the capillary quermassintegrals are expressed through the Euclidean quermassintegrals of the wetting domain. We first recall their normalization on
$\mathcal H\simeq\mathbb R^n$. For a smooth convex body $\Omega\subset\mathbb R^n$, the quermassintegrals of  $\Omega$ are given by
\begin{eqnarray}\label{quermassinterga-Euclidean}
W_{0}^{\mathbb{R}^{n}}(\Omega) &=&{\rm Vol}(\Omega), \nonumber\\ W_{1}^{\mathbb{R}^{n}}(\Omega)&=&\frac{1}{n}{\rm Area}(\partial \Omega), \nonumber\\
W_{k}^{\mathbb{R}^{n}}(\Omega) &=&\frac{1}{n} \binom{n-1}{k-1}^{-1}\int_{\partial\Omega}\sigma^{\mathbb{R}^{n}}_{k-1} dA=\frac{1}{n} \int_{\partial\Omega} H^{\mathbb{R}^{n}}_{k-1} dA, \quad 2 \leq k \leq n.
\end{eqnarray}
In particular,
\begin{equation*}
  W_{n}^{\mathbb{R}^{n}}(\Omega)=\frac{\omega_{n-1}}{n}.
\end{equation*}
The first variation formula of $W_{k}^{\mathbb{R}^{n}}$ is as follows. Consider a family of bounded convex bodies $\left\{\Omega_{t}\right\}$ in $\mathbb{R}^{n}$ whose boundary $\partial \Omega_{t}$ evolving by a normal variation with speed function $f$, then
\begin{equation*}
  \frac{d}{dt} W_{k}^{\mathbb{R}^{n}}(\Omega_{t})=\frac{n-k}{n} \int_{\partial \Omega_{t}} H^{\mathbb{R}^{n}}_{k} f dA_{t}={\binom{n}{k}}^{-1}\int_{\partial \Omega_{t}}  f \sigma^{\mathbb{R}^{n}}_{k}   dA_{t}, \quad 0 \leq k \leq n-1 .
\end{equation*}
For $k=n$, $W_n^{\mathbb R^n}(\Omega_t)=\omega_{n-1}/n$ is constant.

For a smooth bounded wetting domain that is not convex, we use the
same volume and curvature-integral definitions for
$W_l^{\mathbb R^n}$, $0\le l\le n-1$.
Their first variation formula above remains valid. 

Specializing the capillary energy functionals of
Guo--Li--Xia \cite{GLX} to the horosphere, we define %the following capillary quermassintegrals
\begin{eqnarray*}
W_{0,\theta}(\widehat\Sigma)&\coloneqq &\operatorname{Vol}(\widehat\Sigma),\\
W_{1,\theta}(\widehat\Sigma)&\coloneqq &\frac{1}{n+1} \left (\operatorname{Area}(\Sigma)-\cos\theta W_0^{\mathbb R^n}(\widehat{\partial\Sigma})\right). \end{eqnarray*}
For $1\leq k\leq n-1$, set
\begin{eqnarray}\label{Wk-formula} W_{k+1,\theta}(\widehat\Sigma) &\coloneqq &\widetilde W_{k+1,\theta}(\widehat\Sigma) -\frac{k}{n+2-k}W_{k-1,\theta}(\widehat\Sigma),\end{eqnarray}
where \begin{eqnarray*} \widetilde W_{k+1,\theta}(\widehat\Sigma) &=&\frac{1}{n+1}\left\{\int_\Sigma H_kdA
-\cos\theta\sin^k\theta
W_k^{\mathbb R^n}(\widehat{\partial\Sigma})\right.\\
&&\left.-\sum_{l=0}^{k-1}\frac{(-1)^{k+l}}{n-l}
\binom{k}{l}\left[(n-k)\cos^2\theta+k-l\right]
\sin^l\theta\cos^{k-1-l}\theta
W_l^{\mathbb R^n}(\widehat{\partial\Sigma})\right\}.
\end{eqnarray*}
We introduce the highest-order quermassintegral separately by
\begin{eqnarray}\label{Wn-1113}
W_{n+1,\theta}(\widehat\Sigma)
&\coloneqq &\frac{1}{n+1}\left\{\int_\Sigma H_ndA
-\sum_{l=0}^{n-1}(-1)^{n+l}\binom{n}{l}
\sin^l\theta\cos^{n-1-l}\theta
W_l^{\mathbb R^n}(\widehat{\partial\Sigma})\right\}
\nonumber\\
&&-\frac n2 W_{n-1,\theta}(\widehat\Sigma).
\end{eqnarray}
The above normalization implies the following first variation formulae.

\

\begin{theorem}\label{mainthm0.1}
For $n\geq2$ and $\theta\in(0,\pi)$. Let $\Sigma_t$ be a smooth family of properly embedded hypersurfaces
with $\theta$-capillary boundary in $\mathcal H_+$ satisfying
\begin{equation*}
(\partial_t x)^\perp=f\nu.
\end{equation*}
Then for $1\leq k\leq n$,
\begin{equation}\label{eqn:lower-order-1vf}
\frac{d}{dt}W_{k,\theta}(\widehat\Sigma_t)
=\frac{n+1-k}{n+1}\int_{\Sigma_t} f H_k dA_t,
\end{equation}
and
\begin{equation}\label{eqn:top-order-first-variational-formula}
\frac{d}{dt}W_{n+1,\theta}(\widehat\Sigma_t)=0.
\end{equation}
\end{theorem}

\begin{proof}
For $1\leq k\leq n$, Eq. \eqref{eqn:lower-order-1vf} follows from
\cite[Theorem 3.1]{GLX} after specialization and normalization. Next we give a short proof of equation \eqref{eqn:top-order-first-variational-formula} separately. Indeed, 
by Proposition \ref{formula111} (2) and (8), integration by parts, and
Propositions \ref{formula11} and \ref{boundary-equ},  
\begin{eqnarray*}
\frac{d}{dt}\int_{\Sigma_t}H_n dA_t
&=&n\int_{\Sigma_t}H_{n-1}f dA_t +\int_{\partial\Sigma_t}\det(h_{\alpha\beta})
\left(\cot\theta h_{\mu\mu}f-\nabla_\mu f\right)ds_t\\
&=&n\int_{\Sigma_t}H_{n-1}f dA_t
-\frac{1}{\sin\theta}\int_{\partial\Sigma_t}
\det(h_{\alpha\beta})f ds_t.
\end{eqnarray*}
Here $(h_{\alpha\beta})$ is the restriction of $h$ to an
orthonormal frame of $T(\partial\Sigma_t)$; this is the
highest-order specialization of the calculation in
\cite[Appendix A]{GLX}.
The outward normal speed of $\partial\Sigma_t$ in $\mathcal H$
is $\frac{f}{\sin\theta}$. Hence, for $0\leq l\leq n-1$,
\begin{equation*}
\frac{d}{dt}W_l^{\mathbb R^n}(\widehat{\partial\Sigma_t})
=\frac{1}{\sin\theta\binom{n}{l}}
\int_{\partial\Sigma_t} f \sigma_l^{\mathbb R^n}  ds_t.
\end{equation*}
By Proposition \ref{formula11} (2),
\begin{equation*}
\sum_{l=0}^{n-1}(-1)^{n+l}\sin^l\theta\cos^{n-1-l}\theta
\sigma_l^{\mathbb R^n}=-\det(h_{\alpha\beta}).
\end{equation*}
Thus the boundary terms cancel when differentiating
\eqref{Wn-1113}. Applying the already established formula
with $k=n-1$, we obtain
\begin{equation*}
\frac{d}{dt}W_{n+1,\theta}(\widehat\Sigma_t)
=\frac{n}{n+1}\int_{\Sigma_t}H_{n-1}f dA_t
-\frac n2\frac{2}{n+1}\int_{\Sigma_t}H_{n-1}f dA_t=0.
\end{equation*}
\end{proof}
\begin{remark}
    In the Euclidean unit ball, the first variation formulae analogous to \eqref{eqn:lower-order-1vf} and \eqref{eqn:top-order-first-variational-formula} were established in  \cite{SWX} for free boundary hypersurfaces and in \cite{WengX} for capillary hypersurfaces. The corresponding formulae in the Euclidean half-space were obtained in \cite{WWX1}.
\end{remark}

\

\section{A locally constrained curvature flow}\label{sec4}
Let $\Sigma_{0}\coloneqq \Sigma \subset \mathcal{H}_{+}$ be a smooth, properly embedded, strictly convex hypersurface with $\theta$-capillary boundary, given by $x_{0}: \overline{\mathbb{B}}^{n} \rightarrow \mathcal{H}_{+}$. Assume $\Sigma_{t} \subset \mathcal{H}_{+}$ is a family of smooth, properly embedded hypersurfaces with $\theta$-capillary boundary, given by $x(\cdot, t): \overline{\mathbb{B}}^{n} \rightarrow \mathcal{H}_{+}$, evolving by
\begin{equation}\label{equationflow}
\begin{cases}\partial_{t} x(\cdot, t)=f(\cdot, t) \nu(\cdot, t)+\mathcal T(\cdot, t), & \text { in } \mathbb{B}^{n} \times[0, T), \\
\bar{g}(\nu(\cdot, t), \bar{N} \circ x(\cdot, t))=-\cos \theta, & \text { on } \partial \mathbb{B}^{n} \times[0, T), \\ x(\cdot, 0)=x_{0}(\cdot), & \text { in } {\mathbb{B}}^{n},
\end{cases}
\end{equation}
where
\begin{equation}\label{f}
  f\coloneqq\frac{A_{\theta}}{F}-\bar{g}(X,\nu), 
\end{equation}
with
\begin{equation*}
 F\coloneqq\frac{H_k}{H_{k-1}},\qquad 1\leq k\leq n,
\end{equation*}
and with a tangential vector field $\mathcal T\in T\Sigma_t$
satisfying \eqref{boundary2}.

Throughout this section, we assume $\theta\in(0,\pi/2]$ and 
$E_{n+1}\in\operatorname{int}_{\mathcal H}(\widehat{\partial\Sigma_0})$,
and that \eqref{assump} holds for $\Sigma_0$.
For horo-convex initial data, the last condition follows from
Lemma \ref{lemma-horo-umbilical-cap}.

The choice of $f$ in \eqref{f} is motivated by the Minkowski formula
\eqref{mink2}: its pairing with $H_k$ vanishes, whereas the
Newton--MacLaurin inequality gives the sign of its pairing with $H_l$
for $l<k$. The first variation formula then gives the following
preservation and monotonicity properties.
\begin{prop}\label{Wk-Wn}
As long as $\Sigma_{t}$ is strictly convex and the flow \eqref{equationflow} exists with $A_{\theta}>0$, then $W_{k, \theta}$ is preserved and $W_{l, \theta}$ is non-decreasing for $1 \leq l<k \leq n$.
\end{prop}
\begin{proof}
From Theorem \ref{mainthm0.1} and Proposition \ref{mink1}, we know that
\begin{equation*}
\begin{aligned}
\frac{d}{dt} W_{k,\theta}(\widehat{\Sigma}_{t})&=\frac{n+1-k}{n+1} \int_{\Sigma_{t}} H_{k}\left(\frac{A_{\theta}H_{k-1}}{H_{k}}-\bar{g}(X, \nu)\right)dA \\
& = \frac{n+1-k}{n+1} \int_{\Sigma_{t}}\left(A_{\theta}H_{k-1}-H_{k}\bar{g}(X, \nu)\right)dA=0,
\end{aligned}
\end{equation*}
and
\begin{equation*}
\begin{aligned}
\frac{d}{dt} W_{l,\theta}(\widehat{\Sigma}_{t})&=\frac{n+1-l}{n+1} \int_{\Sigma_{t}} H_{l}\left(\frac{A_{\theta}H_{k-1}}{H_{k}}-\bar{g}(X, \nu)\right)dA \\
& \geq \frac{n+1-l}{n+1} \int_{\Sigma_{t}}\left(A_{\theta}H_{l-1}-H_{l}\bar{g}(X, \nu)\right)dA=0,
\end{aligned}
\end{equation*}
where we used the Newton--MacLaurin inequality \eqref{Newton}. This completes the proof.
\end{proof}
\subsection{Scalar formulation and short-time existence}\ 

To express the flow as a scalar parabolic equation, we use radial coordinates centred at $E_{n+1}$. Set 
\begin{equation*}
  \mathbb{S}^{n}_{+}(E_{n+1})\coloneqq  \left\{x\in\mathbb{R}_{+}^{n+1}\mid|x-E_{n+1}|=1~~ \text{and}~~ x_{n+1}>1 \right\}, 
\end{equation*}
and $\mathbb{S}^{n-1}(E_{n+1})\coloneqq \partial\mathbb{S}^{n}_{+}(E_{n+1})$. We employ the polar coordinate $(\rho, \beta, \xi) \in[0,+\infty) \times\left[0, \pi/2 \right] \times \mathbb{S}^{n-1}(E_{n+1})$, where $\xi$ is the spherical coordinate on $\mathbb{S}^{n-1}(E_{n+1})$ and
\begin{equation*}
\begin{cases}{}
\rho^{2}=|X|^{2}=|x'|^{2}+(x_{n+1}-1)^{2}, \\
x_{n+1}=\rho\cos\beta+1,~~|x'|=\rho\sin\beta.
\end{cases}
\end{equation*}
Thus it implies that the standard Euclidean metric has the following expression
\begin{eqnarray*}
    \delta_{\mathbb R^{n+1}_+}
=d\rho^2+\rho^2g_{\mathbb S^n_+(E_{n+1})}
=d\rho^2+\rho^2d\beta^2
+\rho^2\sin^2\beta~g_{\mathbb S^{n-1}(E_{n+1})}.
\end{eqnarray*}

Since $\bar{g}(X, \nu)>0$ on $\Sigma$ by Proposition \ref{bounded}, we can view ${\Sigma}$ as a radial graph over ${\mathbb{S}}_{+}^{n}(E_{n+1})$. Therefore, $\Sigma$ can be written by
\begin{equation*}
X=x-E_{n+1}=\rho(z)(z-E_{n+1}),
\quad z\in\mathbb S^n_+(E_{n+1}),\quad\rho=\rho(\beta,\xi). 
\end{equation*}

In polar coordinates, a direct computation implies that
\begin{eqnarray*}
E_{n+1}=\cos\beta\partial_\rho
-\frac{\sin\beta}{\rho}\partial_\beta. 
\end{eqnarray*}

Let $\nabla^{\mathbb{S}}$ be a Levi-Civita connection with respect to the round metric $g_{\mathbb{S}^{n}_{+}(E_{n+1})}$. Denote
\begin{equation*}
\omega\coloneqq -\log x_{n+1},\qquad u\coloneqq \log\rho,\qquad v\coloneqq \sqrt{1+|\nabla^{\mathbb{S}} u|^{2}}.
\end{equation*}
It is well-known that ${\nu}_{\delta}=\frac{\partial_{\rho}-\rho^{-1} \nabla^{\mathbb{S}} u}{v}$ is the unit outward normal vector of ${\Sigma}$ in $\left(\mathbb{R}_{+}^{n+1}, \delta_{\mathbb{R}_{+}^{n+1}}\right).$
By a straightforward calculation, we get
\begin{eqnarray}\label{bd3}
\qquad\bar{g}(X,\nu)=\bar{g}(x-E_{n+1},\nu)=\frac{1}{x_{n+1}}\left\langle\rho\partial_{\rho},\nu_{\delta}\right\rangle=\frac{1}{x_{n+1}}\left\langle\rho\partial_{\rho},\frac{\partial_{\rho}-\rho^{-1} \nabla^{\mathbb{S}} u}{v}\right\rangle=\frac{1}{v}\rho e^{\omega}
\end{eqnarray}
and
\begin{eqnarray}\label{bd4}
\bar{g}(E_{n+1},\nu)&=&\frac{1}{x_{n+1}}\left\langle E_{n+1},\nu_{\delta}\right\rangle
=\frac{1}{x_{n+1}}\left\langle\cos\beta\partial_{\rho}-\frac{\sin\beta}{\rho}\partial_{\beta},\frac{\partial_{\rho}-\rho^{-1} \nabla^{\mathbb{S}} u}{v}\right\rangle\nonumber\\
&=&\frac{1}{v}e^{\omega}(\cos\beta+\sin\beta\nabla^{\mathbb{S}}_{\beta}u).
\end{eqnarray}
Combining \eqref{bd3} and \eqref{bd4}, we have
\begin{eqnarray*}
\bar{g}(x,\nu)=\bar{g}(X,\nu)+\bar{g}(E_{n+1},\nu)
=\frac1v e^\omega(\rho+\cos\beta+\sin\beta\nabla^\mathbb S_\beta u).  
\end{eqnarray*}
The capillary boundary condition in \eqref{equationflow} gives us that
\begin{equation*}
-\cos\theta=\bar{g}(\nu,\bar{N}\circ x)=-\langle{\nu}_{\delta},E_{n+1}\rangle=\left\langle\frac{\partial_{\rho}-\rho^{-1} \nabla^{\mathbb{S}} u}{v},~\frac{1}{\rho}\partial_{\beta}\right\rangle=-\frac{1}{v}\nabla^{\mathbb{S}}_{\beta}u.
\end{equation*}
It follows that
\begin{equation*}
  \nabla^{\mathbb{S}}_{\partial_{\beta}} u=\cos \theta v \quad \text { on } \partial{\mathbb{S}}_{+}^{n}(E_{n+1}).
\end{equation*}
Thus the flow \eqref{equationflow} is equivalent to the following scalar parabolic equation on $\bar{\mathbb{S}}_{+}^{n}(E_{n+1})$:
\begin{equation}\label{scalarflow}
\begin{cases}{}\partial_{t} u=v{\rho^{-1}}e^{-\omega} {f} &~ \text {in}~~~~ \mathbb{S}_{+}^{n}(E_{n+1}) \times\left[0, T^{*}\right), \\
\nabla^{\mathbb{S}}_{\partial_{\beta}}u=\cos \theta v & \text { in }~~ \partial \mathbb{S}_{+}^{n}(E_{n+1}) \times\left[0, T^{*}\right), \\
u(\cdot, 0)=u_{0}(\cdot) &~~ \text {on}\quad\mathbb{S}_{+}^{n}(E_{n+1}).
\end{cases}
\end{equation}
Here $u_{0}=\log \rho_{0}$ is the corresponding quantity for $\Sigma_{0}$ and ${f}$ is given by

\begin{equation*}
f\coloneqq \frac{e^{\omega}}{vF(h_i^j)}
\left[v-\rho F(h_i^j)
-\cos\theta(\rho+\cos\beta+\sin\beta\partial_\beta u)\right].
\end{equation*}
Here the hyperbolic Weingarten tensor is
\begin{equation*}
h_i^j=\frac{x_{n+1}}{\rho v}
\left[\delta_i^j-
\left((g_{\mathbb S^n_+(E_{n+1})})^{kj}
-\frac{u^ju^k}{v^2}\right)u_{ki}\right]
-\frac{\cos\beta+\sin\beta\nabla^{\mathbb S}_\beta u}{v}\delta_i^j,
\end{equation*}
where indices on $u$ are raised with
$g_{\mathbb S^n_+(E_{n+1})}$ and its covariant derivatives
are taken with respect to $\nabla^{\mathbb S}$.
 
Along $\partial\mathbb S^n_+(E_{n+1})$, set
$J(p)=g_{\mathbb S^n_+(E_{n+1})}(p,\partial_\beta)
-\cos\theta\sqrt{1+|p|^2}$ for tangent vectors $p$.
At a function satisfying the boundary condition in \eqref{scalarflow},
\begin{equation*}
g_{\mathbb S^n_+(E_{n+1})}
(\nabla_pJ|_{p=\nabla^{\mathbb S}u},\partial_\beta)
=1-\cos^2\theta>0.
\end{equation*}
Strict convexity gives $F^{ij}>0$, and Proposition \ref{bounded}
gives $A_\theta>0$. Hence \eqref{scalarflow} is strictly parabolic
with a strictly oblique boundary condition. Short-time existence
follows from the fully nonlinear theory in \cite{GD}  (see also \cite{NU}).
The solution is smooth for $t>0$ and attains the initial data, and 
smoothness up to $t=0$ requires the corresponding higher-order
compatibility conditions.

\subsection{Barrier estimates and star-shapedness}\ 

Let $T^{*}$ be the maximal time on which the solution of
\eqref{equationflow} is smooth, strictly convex and satisfies
$A_{\theta}>0$. These properties hold initially and persist for a
short time. By the initial assumptions, we can choose $0<R_1<R_2$ such that
\begin{equation*}
\widehat{C_{\theta,R_1}(E_{n+1})}
\subseteq\widehat{\Sigma_0}
\subseteq\widehat{C_{\theta,R_2}(E_{n+1})},
\qquad R_2\cdot\cos\theta<1.
\end{equation*}
The estimates below are established on $[0,T^{*})$.

The family of $C_{\theta, r}(E_{n+1})$ forms natural barriers of \eqref{equationflow}. Therefore, we get the following $C^{0}$-estimate and star-shaped estimate:
\begin{prop}%\label{bounded-C0} 
For any $t\in\left[0, T^{*}\right)$, then $\Sigma_{t}$ satisfies
\begin{equation}\label{C0}
\widehat{C_{\theta,R_1}(E_{n+1})}
\subseteq\widehat{\Sigma_t}
\subseteq\widehat{C_{\theta,R_2}(E_{n+1})}.
\end{equation}
In particular, if $\theta\in(0,\pi/2]$ and
$R_{2}\cos\theta<1$, then there exist positive constants
$\epsilon_{0}$, $\epsilon_{1}$ and $\epsilon_{2}$ depending on
$\theta$, $R_{1}$ and $R_{2}$ such that
\begin{eqnarray}
&&1\geq V\geq\epsilon_{0}\label{V-yghs}\\
&&A_{\theta}\geq\epsilon_{1}\label{A-yghs}\\
&&\bar{g}(E_{n+1},\nu)
\geq\frac{\cos\theta}{1+R_{2}(1-\cos\theta)}\geq0\label{E-yghs}\\
&&\bar{g}(X,\nu)\geq\epsilon_{2}.\label{X-yghs}
\end{eqnarray}

\end{prop}
\begin{proof}
From \eqref{dfec}, we have
\begin{equation*}
  \frac{A_{\theta}}{F}=\bar{g}(X,\nu) \qquad\text{ on }~~~ ~~C_{\theta, r}(E_{n+1}).
\end{equation*}
Thus, for every $r>0$ satisfying $r\cos\theta<1$,
$C_{\theta,r}(E_{n+1})$ is a stationary solution of
\eqref{equationflow} in the strictly convex branch. 
The assertion \eqref{C0} follows from the avoidance principle for strictly parabolic equation with capillary boundary conditions (c.f. \cite[Section 2.6]{AWu}). Moreover, the estimates \eqref{V-yghs}-\eqref{X-yghs} are from Proposition \ref{bounded} directly.
\end{proof}

\subsection{Evolution equations along the flow}\ 

We now specialize the variation formulae to flow \eqref{equationflow}. For the subsequent curvature estimates, we write the evolution equations using the linear operator
\begin{equation*}
  \mathcal{L}\coloneqq \partial_{t}-\frac{A_{\theta}}{F^{2}}F^{ij} \nabla_{i j}^{2}-\bar{g}\left(X+\frac{\cos \theta}{F}x+\mathcal T, \nabla\right).
\end{equation*}
Denote $\mathcal{F}\coloneqq \sum_{i=1}^{n} F_{i}^{i}=\sum_{i=1}^{n} \frac{\partial F}{\partial h_{i}^{i}}$.
For $F=\frac{H_{k}}{H_{k-1}}$, we use Proposition \ref{formula2} to obtain
\begin{eqnarray}
\mathcal{F}-\frac{F^{i j} h_{i j}}{F}&=&\mathcal{F}-1 \geq 0, \label{F-110}\\
1\leq\frac{F^{i j} h_{i}^{k} h_{k j}}{F^{2}}&\leq&(n-k+1)\qquad\text{if}~~ 1\leq k\leq n.  \label{F-111}
\end{eqnarray}

\begin{prop} Along flow \eqref{equationflow}, 
\begin{equation}\label{V-eqa}
  \mathcal{L}V=-2A_{\theta}F^{-1}\bar{g}(\nu,E_{n+1})+\bar{g}(X,E_{n+1})-A_{\theta}VF^{-2}\mathcal{F}-\cos\theta F^{-1}\bar{g}(x,\nabla V)
\end{equation}
and
\begin{equation}\label{factor}
\nabla_{\mu}V=\csc\theta V-\cot\theta\bar{g}(\nu,E_{n+1})\qquad\text{ on }~~~\partial\Sigma_{t}.
\end{equation}
\end{prop}
\begin{proof}Along the flow \eqref{equationflow},  
\begin{equation*}
\begin{gathered}
\partial_{t}V=-\bar{g}(\partial_{t}x,E_{n+1})=-\left(A_{\theta}F^{-1}-\bar{g}(X,\nu)\right)\bar{g}(E_{n+1},\nu)-\bar{g}(\mathcal T, E_{n+1}).\\
\end{gathered}
\end{equation*}
From Gauss formula and Proposition \ref{xaa2},  
\begin{equation*}%\label{gauss}
F^{ij}\nabla^{2}_{ij}V=F^{ij}\left(\bar{\nabla}^{2}_{ij}V-h_{ij}\bar{\nabla}_{\nu}V\right)
=VF^{ij}\bar{g}_{ij}+F^{ij}h_{ij}\bar{g}(\nu,E_{n+1}).
\end{equation*}
Together with $F^{ij}h_{ij}=F$ and
$F^{ij}\bar g_{ij}=\mathcal F$, we obtain
\begin{eqnarray*}
 \mathcal{L}V &=&\partial_{t}V-A_{\theta}F^{-2}F^{ij}\nabla^{2}_{ij}V-\bar{g}(X+\cos \theta F^{-1}x+\mathcal T,\nabla V)\\
&=&-2A_{\theta}F^{-1}\bar{g}(\nu,E_{n+1})+\bar{g}(X,E_{n+1})-A_{\theta}VF^{-2}F^{ij}\bar{g}_{ij}-\cos\theta F^{-1}\bar{g}(x,\nabla V)\nonumber\\
&&\quad-\bar{g}(\mathcal T,E_{n+1})-\bar{g}(\mathcal T,\nabla V)\nonumber\\
&=&-2A_{\theta}F^{-1}\bar{g}(\nu,E_{n+1})+\bar{g}(X,E_{n+1})-A_{\theta}VF^{-2}\mathcal{F}-\cos\theta F^{-1}\bar{g}(x,\nabla V).\nonumber
\end{eqnarray*}
Along $\partial\Sigma$, \eqref{mu0} and \eqref{nu0} imply $\bar{N}=\sin \th\mu-\cos \th \nu$, which follows   
\begin{equation*}
\nabla_{\mu}V=-\bar{g}(\mu,E_{n+1})=-\bar{g}(\cot\theta\nu-\csc\theta \bar{E}_{n+1},E_{n+1})=\csc\theta V-\cot\theta\bar{g}(\nu,E_{n+1}).
\end{equation*}
\end{proof}
For the upper bound of $F$, we shall combine its evolution equation with those of $V$ and of the following auxiliary function:
\begin{equation*}
\psi\coloneqq\frac12|X|^2=\frac12|x-E_{n+1}|^2.
\end{equation*}

\begin{prop} Along flow \eqref{equationflow}, %we have
\begin{eqnarray}
\mathcal{L}\psi&=&-|X|^{2}-F^{-2}[2x^{3}_{n+1}A_{\theta}F^{ij}\bar{g}(e_{i},E_{n+1})\bar{g}(e_{j},X)+x_{n+1}A_{\theta}\mathcal{F}] \notag\\
&&+F^{-1}[2A_{\theta}x^{2}_{n+1}\bar{g}(X,\nu)-\cos\theta x^{2}_{n+1}\bar{g}(X,x)+\cos\theta x^{2}_{n+1}\bar{g}(X,\nu)\bar{g}(x,\nu)]  , \label{Phsi}
\end{eqnarray}
and
\begin{equation}\label{boundarypsi}
  \nabla_{\mu} \psi=\cot\theta \bar{g}(X,\nu)\qquad \text{ on } ~~~\partial\Sigma_{t}.
\end{equation}
\end{prop}
\begin{proof}
From \eqref{equationflow} and \eqref{f},  
\begin{equation*}
  \partial_{t}\psi=\langle\partial_{t}X,X\rangle=\langle\partial_{t}x,X\rangle=f\langle X,\nu\rangle+\langle X,\mathcal T\rangle=(A_{\theta}F^{-1}-\bar{g}(X,\nu))\langle X,\nu\rangle+\langle X,\mathcal T\rangle.
\end{equation*}
Choosing $\{e_{i}\}_{i=1}^{n}$ be an orthonormal frame of $\Sigma$ such that $\nabla_{e_{i}}e_{j}=0$ at $p\in\Sigma$, thus
\begin{eqnarray*}
F^{ij}\nabla_{ij}\psi&=&F^{ij}\left(\frac12|X|^{2} \right)_{;ij}
=F^{ij}[x_{n+1}^{2}\bar{g}(e_i,X)]_{;j}\\
&=&2x^{3}_{n+1}F^{ij}\bar{g}(e_{i},E_{n+1})\bar{g}(e_{j},X)+x^{2}_{n+1}F^{ij}(\bar{g}(\bar{\nabla}_{e_{j}}e_{i},X)+\bar{g}(\bar{\nabla}_{e_{j}}X,e_{i}))\nonumber\\
&=&2x^{3}_{n+1}F^{ij}\bar{g}(e_{i},E_{n+1})\bar{g}(e_{j},X)-x^{2}_{n+1}\bar{g}(X,\nu)F^{ij}h_{ij}+x_{n+1}F^{ij}\bar{g}_{ij}\nonumber\\
&=&2x^{3}_{n+1}F^{ij}\bar{g}(e_{i},E_{n+1})\bar{g}(e_{j},X)-x^{2}_{n+1}\bar{g}(X,\nu)F+x_{n+1}\mathcal{F}.\nonumber
\end{eqnarray*}
By calculating, we see
\begin{equation*}
  \nabla\psi=\langle X,e_{i}\rangle e_{i}=x^{2}_{n+1}(X-\bar{g}(X,\nu)\nu).
\end{equation*}
It follows that
\begin{eqnarray*}
\mathcal{L}\psi&=&\partial_{t}\psi-A_{\theta}F^{-2}F^{ij}\nabla_{ij}\psi-\bar{g}(X+\cos\theta F^{-1}x+\mathcal T,\nabla\psi)\\
&=&-|X|^{2}-F^{-2}[2x^{3}_{n+1}A_{\theta}F^{ij}\bar{g}(e_{i},E_{n+1})\bar{g}(e_{j},X)+x_{n+1}A_{\theta}\mathcal{F}]\nonumber\\
&&+F^{-1}[2A_{\theta}x^{2}_{n+1}\bar{g}(X,\nu)-\cos\theta x^{2}_{n+1}\bar{g}(X,x)+\cos\theta x^{2}_{n+1}\bar{g}(X,\nu)\bar{g}(x,\nu)].\nonumber
\end{eqnarray*}
Along $\partial\Sigma_{t}\subseteq \mathcal{H}$, we use \eqref{mu0} and \eqref{XXaeq2} to get
\begin{equation*}
\nabla_{\mu}\psi=\langle D_{\mu}X,X\rangle=\langle\mu,X\rangle=\csc\theta\langle \bar{N}+\cos\theta\nu,X\rangle=\cot\theta \bar{g}(X,\nu).
\end{equation*}
\end{proof}

\begin{prop}\label{Festimate} Along flow \eqref{equationflow}, 
\begin{equation}\label{F88}
\mathcal{L}F= 2{F^{-2}}F^{i j}A_{\theta; i} F_{;j}-2F^{-3}A_{\theta}F^{i j}F_{;i}F_{;j}+(\mathcal{F}-1)\bar{g}(E_{n+1}, \nu)+VF^{-1}(F^{2}-F^{ij}(h^{2})_{ij})
\end{equation}
and
\begin{equation}\label{boundaryF}
  \nabla_{\mu} F=0 \qquad\text{ on }~~~ \partial\Sigma_{t}.
\end{equation}
\end{prop}
\begin{proof} From Proposition \ref{formula111}, we get
\begin{eqnarray}\label{F89}
\partial_{t}F&=&F^{j}_{i}\partial_{t}h^{i}_{j}\nonumber\\
&=&-F^{ij}f_{ij}-fF^{ij}(h^{2})_{ij}+fF^{ij}\bar{g}_{ij}+\nabla_{\mathcal T}(F^{ij}h_{ij})\nonumber\\
&=&-F^{ij}\left(\frac{A_{\theta}}{F}-\bar{g}(X,\nu)\right)_{;ij}-fF^{ij}(h^{2})_{ij}+fF^{ij}\bar{g}_{ij}+\nabla_{\mathcal T}F.
\end{eqnarray}
Next we compute the first term on the right-hand side above.
\begin{eqnarray*}
-F^{ij}\left(\frac{A_{\theta}}{F}-\bar{g}(X,\nu)\right)_{;ij}&=&-F^{ij}\left(\frac{A_{\theta}}{F}\right)_{;ij}+F^{ij}\bar{g}(X,\nu)_{;ij}\nonumber\\
&=&A_{\theta}F^{-2}F^{ij}F_{;ij}-2A_{\theta}F^{-3}F^{ij}F_{;i}F_{;j}\nonumber\\
&&+2F^{-2}F^{ij}F_{;i}A_{\theta;j}-F^{-1}F^{ij}A_{\theta;ij}+F^{ij}\bar{g}(X,\nu)_{;ij}\nonumber
\end{eqnarray*}
By \eqref{Va-1}, we calculate that
\begin{eqnarray}\label{F113}
A_{\theta;ij}&=&\nabla^{2}_{ij}V-\cos\theta \bar{g}(x,\nu)_{;ij}\nonumber\\
&=&\bar{\nabla}^{2}_{ij}V-h_{ij}\bar{\nabla}_{\nu}V-\cos\theta \bar{g}(x,\nu)_{;ij}\nonumber\\
&=&V\bar{g}_{ij}+h_{ij}\bar{g}(E_{n+1},\nu)-\cos\theta \bar{g}(x,\nu)_{;ij}
\end{eqnarray}
Since $x$ is a Killing vector field in $\mathbb{H}^{n+1}$, by \eqref{xKilling},
\begin{eqnarray*}
 e_{i}(\bar{g}(x,\nu))=\bar{g}(x,\bar{\nabla}_{e_{i}}\nu)+\bar{g}(\bar{\nabla}_{e_{i}}x,\nu)=\bar{g}(x,\bar{\nabla}_{e_{i}}\nu)-\bar{g}(\bar{\nabla}_{\nu}x, e_{i}).
\end{eqnarray*}
Calculating further, we get
\begin{eqnarray}\label{F119}
\bar{g}(x,\nu)_{;ij}&=&e_{j}e_{i}\bar{g}(x,\nu)\nonumber\\
&=&\bar{g}(x,\bar{\nabla}_{e_{j}}\bar{\nabla}_{e_{i}}\nu)+\bar{g}(\bar{\nabla}_{e_{j}}x,\bar{\nabla}_{e_{i}}\nu)-\bar{g}(\bar{\nabla}_{\nu}x,\bar{\nabla}_{e_{j}}e_{i})
-\bar{g}(\bar{\nabla}_{e_{j}}\bar{\nabla}_{\nu}x,e_{i})\nonumber\\
&=&-\bar{g}(\bar{\nabla}_{\nu}\bar{\nabla}_{e_{j}}x,e_{i})+\bar{g}(\bar{\nabla}_{[\nu,e_{j}]}x,e_{i})+\bar{g}(\bar{R}(\nu,e_{j})x,e_{i})+h_{ij}\bar{g}(\bar{\nabla}_{\nu}x,\nu)\nonumber\\
&&+\bar{g}(\bar{\nabla}_{e_{j}}x,e_{k})h_{ik}+\bar{g}(x,h_{ij;k}e_{k})-h_{ik}h^{k}_{j}\bar{g}(x,\nu)\nonumber\\
&=&-\bar{\nabla}_{\nu}(\bar{g}(\bar{\nabla}_{e_{j}}x,e_{i}))+\bar{g}(\bar{\nabla}_{e_{j}}x,\bar{\nabla}_{\nu}e_{i})+\bar{g}(\bar{\nabla}_{[\nu,e_{j}]}x,e_{i})+\bar{g}(x,\nu)\bar{g}_{ij}\nonumber\\
&&+h_{ij}\bar{g}(\bar{\nabla}_{\nu}x,\nu)+\bar{g}(\bar{\nabla}_{e_{j}}x,e_{k})h_{ik}+\bar{g}(x,h_{ij;k}e_{k})-(h^{2})_{ij}\bar{g}(x,\nu)\nonumber\\
&=&-\bar{\nabla}_{\nu}(\bar{g}(\bar{\nabla}_{e_{j}}x,e_{i}))+\bar{g}(\bar{\nabla}_{e_{j}}x,[\nu,e_{i}])+\bar{g}(\bar{\nabla}_{[\nu,e_{j}]}x,e_{i})+\bar{g}(x,\nu)\bar{g}_{ij}\nonumber\\
&&+h_{ij}\bar{g}(\bar{\nabla}_{\nu}x,\nu)+2\bar{g}(\bar{\nabla}_{e_{j}}x,e_{k})h_{ik}+\bar{g}(x,h_{ij;k}e_{k})-(h^{2})_{ij}\bar{g}(x,\nu),
\end{eqnarray}
which yields that
\begin{eqnarray*}
F^{ij}A_{\theta;ij}&=&VF^{ij}\bar{g}_{ij}+F^{ij}h_{ij}\bar{g}(E_{n+1},\nu)-\cos\theta F^{ij}\bar{g}(x,\nu)_{;ij}\\&=&A_{\theta}\mathcal F+F\bar g(E_{n+1},\nu)-\cos\theta\bar g(x,\nabla F)+\cos\theta F^{ij}h_{ik}h^k_j\bar g(x,\nu).\end{eqnarray*}Here we used $F^{ij}h_{ij}=F$ and the symmetry of $F^{ij}h_{ik}$,
which makes its contraction with the skew-symmetric tensor
$\bar g(\bar\nabla_{e_j}x,e_k)$ vanish.
 
Similar to \eqref{F119}, we can also compute to the Hessian of $\bar{g}(X,\nu)$ as
\begin{eqnarray}\label{F119-X}
\bar{g}(X,\nu)_{;ij}&=&-\bar{\nabla}_{\nu}(\bar{g}(\bar{\nabla}_{e_{j}}X,e_{i}))+\bar{g}(\bar{\nabla}_{e_{j}}X,[\nu,e_{i}])+\bar{g}(\bar{\nabla}_{[\nu,e_{j}]}X,e_{i})+\bar{g}(X,\nu)\bar{g}_{ij}\nonumber\\
&&+h_{ij}\bar{g}(\bar{\nabla}_{\nu}X,\nu)+2\bar{g}(\bar{\nabla}_{e_{j}}X,e_{k})h_{ik}+\bar{g}(X,h_{ij;k}e_{k})-(h^{2})_{ij}\bar{g}(X,\nu).
\end{eqnarray}
Contracting \eqref{F119-X} and using \eqref{XXaeq1}, together with
$F^{ij}h_{ij}=F$ and $F^{ij}\bar g_{ij}=\mathcal F$, we obtain
\begin{eqnarray*}
F^{ij}\bar{g}(X,\nu)_{;ij}&=&(\bar{g}(X,\nu)-\bar{\nabla}_{\nu}V)F^{ij}\bar{g}_{ij}+VF^{ij}h_{ij}+F^{ij}\bar{g}(X,h_{ij;k}e_{k})\\
&&-F^{ij}(h^{2})_{ij}\bar{g}(X,\nu)\nonumber\\
&=&\bar{g}(x,\nu)\mathcal{F}+VF+\bar{g}(X,\nabla F)-F^{ij}(h^{2})_{ij}\bar{g}(X,\nu),\nonumber
\end{eqnarray*}
where we use the identity 
\begin{equation*}
 \bar{g}(e_{j},[\nu,e_{i}])+\bar{g}(e_{i},[\nu,e_{j}])+\bar{g}(\nu,[e_{i},e_{j}])=-2\bar{g}(\bar{\nabla}_{e_{i}}\nu, e_{j})=-2h_{ij}.
\end{equation*}
Moreover,  
\begin{equation*}
 -f F^{ij}(h^{2})_{ij}=-\left(A_{\theta}F^{-1}-\bar{g}(X, \nu)\right)F^{ij}(h^{2})_{ij}.
\end{equation*}
Putting those terms into equation \eqref{F89} and simplifications, we conclude that
\begin{eqnarray*}
\partial_{t} F&=&A_{\theta}F^{-2} F^{ij} F_{;i j}+F_{; k}\bar{g}(X, e_{k})+2 {F^{-2}}F^{i j} F_{;i}A_{\theta;j}-2A_{\theta}{F^{-3}} F^{i j} F_{;i} F_{; j}+\nabla_{\mathcal T} F \\
&&+\cos \theta F^{-1} F_{;k}\bar{g}(x,e_{k})+\left(\mathcal{F}-1\right)\bar{g}(E_{n+1},\nu)
+VF^{-1}(F^{2}-F^{ij}(h^{2})_{ij}),\nonumber
\end{eqnarray*}
Hence, we have
\begin{eqnarray*}
 \mathcal{L}F&=&\partial_{t}F-A_{\theta}F^{-2}F^{ij}\nabla^{2}_{ij}F-\bar{g}(X+\cos\theta{F^{-1}}x+\mathcal T,\nabla F)\\
&=& 2{F^{-2}}F^{i j}A_{\theta; i} F_{;j}-2A_{\theta}F^{-3}F^{i j}F_{;i}F_{;j}+(\mathcal{F}-1)\bar{g}(E_{n+1}, \nu)\nonumber\\
&&+VF^{-1}(F^{2}-F^{ij}(h^{2})_{ij}).\nonumber
\end{eqnarray*}

Next we show \eqref{boundaryF}. From Proposition \ref{boundary-equ},  
\begin{equation*}
  \nabla_{\mu} f=\left(\csc\theta+\cot \theta h_{\mu \mu}\right) f .
\end{equation*}
From \cite[Proposition 3.6]{GWX}, along $\partial \Sigma_{t}$,
\begin{equation*}
  \nabla_{\mu}\bar{g}(X, \nu)=\left(\csc\theta+\cot \theta h_{\mu \mu}\right)\bar{g}(X, \nu),
\end{equation*}
and
\begin{equation*}
  \nabla_{\mu}A_{\theta}=\left(\csc\theta+\cot \theta h_{\mu \mu}\right)A_{\theta}.
\end{equation*}
Together with \eqref{f} and $A_{\theta}>0$ by \eqref{A-yghs}, then the assertion\eqref{boundaryF} follows.
\end{proof}

\begin{prop} Along flow \eqref{equationflow}, 
\begin{eqnarray}\label{H1}
\mathcal{L}H&=&(n-F^{-1}H)\bar{g}(E_{n+1},\nu)
+A_{\theta}F^{-2}\bar{g}^{ij}F^{kl;rs}h_{kl;i}h_{rs;j}
+VH+A_{\theta}F^{-2}H F^{ij}(h^{2})_{ij}\nonumber\\
&&-\left(A_{\theta}+V\right)|h|^{2}F^{-1}
-2A_{\theta}F^{-3}|\nabla F|^{2}
+A_{\theta}F^{-2}H\mathcal{F}-nA_{\theta}F^{-1}\nonumber\\
&&+2F^{-2}\bar{g}(\nabla A_{\theta},\nabla F),
\end{eqnarray}
and when $\theta\in(0,\pi/2]$,
\begin{equation}\label{boudH}
  \nabla_{\mu} H \leq 0 \quad\text{ on } ~~~\partial \Sigma_{t}.
\end{equation}
\end{prop}
\begin{proof} By Proposition \ref{formula111}, 
\begin{eqnarray}\label{H00}
\partial_{t} H&=&-\Delta f+(n-|h|^{2}) f+\nabla_{\mathcal T} H \nonumber\\
&=&-\Delta\left(A_{\theta}{F^{-1}}-\bar{g}(X,\nu)\right)+(n-|h|^{2})\left(A_{\theta}{F^{-1}}-\bar{g}(X,\nu)\right)+\nabla_{\mathcal T}H.
\end{eqnarray}
Note that, from \eqref{F119} and \eqref{F119-X}, we derive
\begin{equation}\label{Laplacian}
  \Delta \bar{g}(X,\nu)=n\bar{g}(x,\nu)+VH+\bar{g}(X,\nabla H)-|h|^{2}\bar{g}(X,\nu)
\end{equation}
and
\begin{eqnarray*}
\Delta\left(VF^{-1}\right)&=&2VF^{-3}|\nabla F|^{2}-VF^{-2}\Delta F+2F^{-2}\bar{g}(E_{n+1},\nabla F)\\
&&+F^{-1}H\bar{g}(E_{n+1},\nu)+nVF^{-1}\nonumber
\end{eqnarray*}
and
\begin{eqnarray}\label{fun2}
\Delta\left(\bar{g}(x,\nu)F^{-1}\right)
&=&2F^{-3}|\nabla F|^{2}\bar{g}(x,\nu)-F^{-2}\Delta F\bar{g}(x,\nu)-2F^{-2}F_{;i}\bar{g}(x,\nu)_{;i}+F^{-1}\Delta\bar{g}(x,\nu)\nonumber\\
&=&2F^{-3}|\nabla F|^{2}\bar{g}(x,\nu)-F^{-2}\Delta F\bar{g}(x,\nu)-2F^{-2}F_{;i}\bar{g}(x,\nu)_{;i}\nonumber\\
&&+F^{-1}\left(n\bar{g}(x,\nu)+\bar{g}(x,\nabla H)-|h|^{2}\bar{g}(x,\nu)\right)
\end{eqnarray}
Inserting \eqref{Laplacian} and \eqref{fun2} into \eqref{H00}, we obtain
\begin{eqnarray}\label{Fij}
\partial_{t} H&=&n\bar{g}(E_{n+1},\nu)+VH+\bar{g}(X,\nabla H)-F^{-1}|h|^{2}V+\nabla_{\mathcal T} H\nonumber\\
&&-2 V F^{-3}|\nabla F|^{2}+VF^{-2}\Delta F-2F^{-2}\bar{g}(E_{n+1},\nabla F)-F^{-1}H\bar{g}(E_{n+1},\nu)\nonumber\\
&&+2\cos\theta F^{-3}|\nabla F|^{2}\bar{g}(x,\nu)-\cos\theta F^{-2}\Delta F\bar{g}(x,\nu)\nonumber\\
&&-2\cos\theta F^{-2}F_{;i}\bar{g}(x,\nu)_{;i}+\cos\theta F^{-1}\bar{g}(x,\nabla H).
\end{eqnarray}

Next we compute the term containing $\Delta F$. Due to the Gauss-Codazzi equations and the Ricci identities in $\mathbb{H}^{n+1}$,  
\begin{eqnarray}\label{Fij2}
h_{k l ; i j}&=&h_{i j ; k l}+\left(\bar{R}_{aklj}+h_{l a} h_{j k}-h_{l k} h_{j a}\right) h_{i}^{a}+\left(\bar{R}_{ailj}+h_{l a} h_{j i}-h_{l i} h_{j a}\right) h_{k}^{a}\nonumber\\
&=&h_{i j ; k l}+\left(h_{l a} h_{j k}-h_{l k} h_{j a}\right) h_{i}^{a}+\left(h_{l a} h_{j i}-h_{l i} h_{j a}\right) h_{k}^{a}\nonumber\\
&&-(h_{il}\bar{g}_{kj}-h_{ij}\bar{g}_{kl})-(h_{kl}\bar{g}_{ij}-h_{jk}\bar{g}_{il}).
\end{eqnarray}
Consequently, 
\begin{equation*}
  \bar{g}^{i j} F^{k l} h_{k l ; i j}=F^{k l} H_{; k l}+F^{k l} h_{l a} h_{k}^{a} H-F|h|^{2}+HF^{ij}\bar{g}_{ij}-nF^{ij}h_{ij},
\end{equation*}
and 
\begin{eqnarray*}
\Delta F &=&\bar{g}^{i j}\left(\frac{\partial^{2} F}{\partial h_{k l} \partial h_{rs}} h_{k l ; i} h_{rs ; j}+\frac{\partial F}{\partial h_{k l}} h_{k l ; i j}\right) \\
&=&\bar{g}^{i j} F^{kl;rs}h_{k l ; i} h_{rs, j}+F^{k l} H_{; k l}+H F^{k l} h_{l a} h_{k}^{a}-F|h|^{2}+HF^{ij}\bar{g}_{ij}-nF^{ij}h_{ij}.\nonumber
\end{eqnarray*}
By inserting the above terms into equation \eqref{Fij}, we get the evolution equation for $H$.

Finally, we turn to prove $\nabla_{\mu} H \leq 0$ on $\partial \Sigma_{t}$. Let $\left\{e_{\alpha}\right\}_{\alpha=1}^{n-1}$ be an orthonormal basis of $T(\partial \Sigma_{t})$ such that the second fundamental form $h$ of $\Sigma_{t}$ is diagonal. From \eqref{boundaryF}, we have
\begin{equation}\label{Fmumu}
0=\nabla_{\mu} F=F^{\mu \mu} h_{\mu \mu;\mu}+\sum_{\alpha=1}^{n-1} F^{\alpha \alpha} h_{\alpha \alpha; \mu} .
\end{equation}
Together with Proposition \ref{formula11} (4), we derive
\begin{eqnarray*}
\nabla_{\mu} H &=&h_{\mu \mu; \mu}+\sum_{\alpha=1}^{n-1} h_{\alpha \alpha;\mu}=-\sum_{\alpha=1}^{n-1}\frac{F^{\alpha\alpha}}{F^{\mu\mu}}
 h_{\alpha\alpha;\mu}
+\sum_{\alpha=1}^{n-1}h_{\alpha\alpha;\mu}\nonumber\\
&=&\sum_{\alpha=1}^{n-1}\frac1{F^{\mu\mu}}
(F^{\mu\mu}-F^{\alpha\alpha})(h_{\mu\mu}-h_{\alpha\alpha})
\tilde h_{\alpha\alpha}.\nonumber
\end{eqnarray*}
Since $F$ is a concave function, then 
\begin{equation*}
 \left(F^{\mu \mu}-F^{\alpha \alpha}\right)\left(h_{\mu \mu}-h_{\alpha \alpha}\right) \leq 0, \quad \forall\,  1\leq \alpha \leq n-1. 
\end{equation*}
Due to Corollary \ref{cor222}, we know that $\tilde{h}_{\alpha \alpha} \geq 0$. It holds that
\begin{equation*}
\nabla_{\mu} H=\sum_{\alpha=1}^{n-1} \frac{1}{F^{\mu \mu}}\left(F^{\mu \mu}-F^{\alpha \alpha}\right)\left(h_{\mu \mu}-h_{\alpha \alpha}\right) \tilde{h}_{\alpha \alpha} \leq 0 .
\end{equation*}
This completes the proof.
\end{proof}

\section{Curvature estimates and convergence}\label{curvature-sec} 

For horo-convex initial data, we consider $F=H_k/H_{k-1}$ with $1\leq k\leq n$. For general strictly convex initial data, we take $F=H_n/H_{n-1}$. We work on the interval $[0,T^*)$ introduced above and retain the assumptions of Section \ref{sec4}.
%In this section we always assume that $F=\frac{H_{k}}{H_{k-1}}$ with $1\leq k\leq n$ for the initial horo-convex hypersurface and $F=\frac{H_{n}}{H_{n-1}}$ for the initial convex hypersurface.
\subsection{Horo-convexity preservation and lower curvature bounds}\ %\label{sec-horo-convexity}

We use a boundary tensor maximum principle to show that $h-g\geq0$ is preserved along the flow. After recalling the boundary tensor maximum principle, we establish the geometric estimates needed to verify its null-vector condition. For general strictly convex initial data with $F=H_n/H_{n-1}$, we shall instead obtain a lower curvature bound from the scalar evolution equation for $F$.

The boundary tensor maximum principle stated below is due to Hu--Wei--Yang--Zhou \cite[Theorem 1.2]{HWYZhalf} and refines Stahl's result \cite[Theorem 3.3]{Stahl}. Both results can be viewed as boundary counterparts of the tensor maximum principle on compact manifolds, originally established by Hamilton \cite[Theorem 9.1]{Hamilton1982} and subsequently refined by Andrews \cite[Theorem 3.2]{Andrews2007}.

%We begin by recalling the following boundary tensor maximum principle from Hu-Wei-Yang-Zhou \cite[Theorem 1.2]{HWYZhalf}, which refines a result of Stahl \cite[Theorem 3.3]{Stahl}. Both can be viewed as boundary counterparts of the tensor maximum principle on compact manifolds, originally established by Hamilton \cite[Theorem 9.1]{Hamilton1982} and later refined by Andrews \cite[Theorem 3.2]{Andrews2007}.

%The boundary tensor maximum principle of Hu--Wei--Yang--Zhou \cite[Theorem 1.2]{HWYZhalf} incorporates the refined null-vector condition of Andrews \cite[Theorem 3.2]{Andrews2007}. It strengthens the boundary result of Stahl \cite[Theorem 3.3]{Stahl}; for the classical closed-manifold principle, see Hamilton \cite[Theorem 9.1]{Hamilton1982}.

\begin{theorem}[\cite{HWYZhalf}]\label{thm:tensor-m.p.} 
  Let $M$ be a smooth compact Riemannian manifold with boundary, and let $\mu$ be its outward unit normal vector. Suppose that $S(t)$ is a smooth symmetric $(0,2)$-tensor field satisfying
	\begin{eqnarray*}
	\partial_t S =a^{ij}\nabla_i\nabla_j S+ b^k\nabla_kS+N,\end{eqnarray*}
where $(a^{ij})$ is positive definite, all coefficients are smooth, and $\nabla$ is a smooth symmetric connection. Suppose that the smooth symmetric $(0,2)$-tensor $N(t)$ satisfies
	\begin{eqnarray}\label{tensor MP 1}
 N(\xi,\xi)+\sup\limits_{\Lambda}  2a^{kl} \left(2\Lambda^p_k \nabla_l S_{ip} \xi^i-\Lambda^p_k\Lambda^q_l S_{pq} \right) \geq 0, ~~  	\text{ in }  M\times (0,T],
\end{eqnarray}  and
\begin{eqnarray*}
(\nabla_\mu S_{ij})\xi^i\xi^j \geq 0, ~~ \text{ on } \partial M\times (0,T] ,
	\end{eqnarray*} whenever $S\geq0$ and $S_{ij}\xi^j=0$. If $S(0)\geq0$, then $S(t)\geq0$ on $M$ for all $0\leq t\leq T$.
 \end{theorem}

To verify the null-vector condition, we first establish the
geometric bounds entering its zeroth-order and gradient terms. The next lemma also shows that the containment assumption in Section \ref{sec4} is automatic for horo-convex initial data.
\begin{lemma}\label{lemma-horo-umbilical-cap}
For $\theta\in(0,\pi/2]$. If $\Sigma$ is horo-convex hypersurface with $\theta$-capillary boundary in $\mathcal{H}_{+}$ and  $E_{n+1}\in\operatorname{int}_{\mathcal H}(\widehat{\partial\Sigma}) $, then there exists a constant $R>0$ such that
\begin{equation*}
  \widehat{\Sigma}\subseteq \widehat{C_{\theta,R}(E_{n+1})}\quad \text{and}\quad R\cdot\cos\theta<1.
\end{equation*}
\end{lemma}
\begin{proof}
Fix $p=(p',1)\in\partial\Sigma$. From \eqref{mu0} and \eqref{nu0} we have
\begin{equation*}
 \nu=\sin\theta\bar{\nu}+\cos\theta \bar{E}_{n+1},\qquad
 \mu=\cos\theta\bar{\nu}-\sin\theta \bar{E}_{n+1}.
\end{equation*}
Denote
\begin{equation*}
 \rho\coloneqq \frac1{1+\cos\theta},\qquad q'\coloneqq p'-\rho \sin\theta\bar{\nu}\in\mathbb{R}^{n},\qquad q\coloneqq (q',\rho).
\end{equation*}
Hence the Euclidean ball $\mathbb{B}_\rho(q)$ is tangent to $\partial_{\infty}\mathbb{H}^{n+1}=\{x_{n+1}=0\}$, which is a
hyperbolic horoball. Moreover,
\begin{equation*}%\label{wsxq}
      \frac{\nu}{x_{n+1}}(p)=\nu_{\delta}(p)=\frac{p-q}{\rho} .
\end{equation*}
Since $\Sigma$ is a horo-convex hypersurface, the wetting face and its edge are supported by the horoball $\overline{\mathcal H_+}$. The convex body $\widehat\Sigma$ is $1$-convex in the barrier sense of \cite[Definition 2.1]{DrachTatarko}. Therefore, Blaschke's rolling theorem \cite[Theorem 2.2]{DrachTatarko} implies a globally supporting horoball at every interior point of $\Sigma$. Taking the limit as such points approach $p$ shows that $\widehat\Sigma\subseteq\overline{\mathbb B_\rho(q)}$. Hence
\begin{equation*}%\label{eq:offset-cap}
  \widehat{\Sigma}\subseteq \overline{\mathbb{B}_\rho(q)}\cap\overline{\mathcal H_+} , 
\end{equation*}
where the horizontal center $q'$ need not be zero. Since $E_{n+1}\in\operatorname{int}_{\mathcal H}(\widehat{\partial\Sigma}) $, we have
\begin{equation*}%\label{eq:center-distance}
 |q'|^2+(1-\rho)^2\le\rho^2,
 \qquad 1-\rho=\rho \cos\theta,
 \qquad |q'|\le\rho \sin\theta.
\end{equation*}
Let $R=2\rho$ and ${O}=(0,1-R\cdot\cos\theta)$. Then
\begin{eqnarray*}
 |q-O|^2=|q'|^2+\rho^2\cos^2\theta\le\rho^2.
\end{eqnarray*}

For each point $\hat{p}\in\widehat{\Sigma}$, the triangle inequality gives
$|\hat{p}-O|\le|\hat{p}-q|+|q-O|\le2\rho=R$. Thus
\begin{equation*}%\label{eq:offset-cap-jdidnc}
 \widehat{\Sigma}\subseteq \widehat{{C_{\theta,R}(E_{n+1})}},
\end{equation*}
and
\begin{equation*}%\label{eq:R-positive}
 R\cdot\cos\theta=\frac{2\cos\theta}{1+\cos\theta}<1,
 \qquad \kappa_R=\frac{1}{R}-\cos\theta=\frac{1-\cos\theta}{2}>0.
\end{equation*}
This completes the proof.
\end{proof}

As a direct consequence of Lemma \ref{lemma-horo-umbilical-cap} and Proposition \ref{bounded}, we obtain the results below.
\begin{corollary}\label{horo-bounded} Let $\Sigma$ be a horo-convex hypersurface with $\theta$-capillary boundary for $\theta \in\left(0, \pi/2 \right]$ in $\mathcal{H}_{+}$. Assume that $E_{n+1}\in\operatorname{int}_{\mathcal H}(\widehat{\partial\Sigma}) $. Then there exist positive constants $\varepsilon_{1}$, $\varepsilon_{2}$,   such that
\begin{itemize}
  \item [{\rm(i)}]$1\geq V\geq\varepsilon_{1}$;
  \item [{\rm(ii)}]$A_{\theta}\geq \varepsilon_{2}$;
  \item [{\rm(iii)}]$\bar g(E_{n+1},\nu)\geq0$.
\end{itemize}
\end{corollary}
Besides these positive lower bounds for $V$ and $A_\theta$,
we need to control $V-\bar\nabla_\nu V$. The next estimate will enter the geometric inequality for the gradient terms in the evolution of $h-g$.
\begin{lemma}\label{horo-bate} Let $\Sigma$ be a horo-convex hypersurface with $\theta$-capillary boundary for $\theta \in\left(0, \pi/2 \right]$ in $\mathcal{H}_{+}$. Assume that $E_{n+1}\in\operatorname{int}_{\mathcal H}(\widehat{\partial\Sigma}) $. Then 
\begin{equation*}%\label{njkopk}
 V-\bar{\nabla}_{\nu}V\geq1+\cos\theta>0.
\end{equation*}
\end{lemma}
\begin{proof}
By Corollary \ref{horo-bounded} (iii),
$\bar\nabla_\nu V=-\bar g(E_{n+1},\nu)\leq0$.
Hence
\begin{equation}\label{eq:V-convex}
 \nabla^2V=Vg-(\bar\nabla_\nu V)h
 \geq Vg\geq\varepsilon_1g.
\end{equation}

Now we start the forward integral curve $\gamma(t)$ of $\nabla V$ at a noncritical point $p$.
If $\gamma(t)$ stays in $\Sigma$ for all $t>0$, then
\begin{equation*}%\label{bgyyyb}
 \frac{d}{dt}g(\nabla V,\nabla V)
 =2\nabla^2V(\nabla V,\nabla V)
 \ge2\varepsilon_{1}g(\nabla V,\nabla V),
\end{equation*}
which contradicts the boundedness of $|\nabla V|$ on the compact hypersurface. Therefore, $\gamma(t)$ reaches the boundary $\partial\Sigma$ in finite time. Along $\gamma(t)$,
\begin{equation*}%\label{jkjknv}
\frac{d}{dt}(V-\bar{\nabla}_{\nu}V)=-(h-g)(\nabla V,\nabla V)\le0,
\end{equation*}
where we used \eqref{Va-1}. Therefore
\begin{equation*}
 V-\bar{\nabla}_{\nu}V\geq(V-\bar{\nabla}_{\nu}V)|_{\partial\Sigma}=1+\cos\theta>0,
\end{equation*}
at every noncritical point. The Hessian in \eqref{eq:V-convex} is positive definite, so critical points are isolated.
Hence the proof is finished.
\end{proof}

We combine this estimate with a supporting horoball to bound the derivative of $A_\theta$ in a principal direction of curvature one. These are precisely the null directions of $h-g$. 
\begin{lemma}\label{lem:geometric-gradient}
For $\theta \in\left(0, \pi/2 \right]$. Let $\Sigma$ be a compact horo-convex capillary hypersurface in $\mathcal{H}_{+}$. Assume that $E_{n+1}\in\operatorname{int}_{\mathcal H}(\widehat{\partial\Sigma})$ and $e_{1}$ is a principal direction with curvature $1$. Then
\begin{equation*}%\label{eq:shifted-gap}
 2A_{\theta}(V-\bar{\nabla}_{\nu}V)-|\bar{\nabla}_{e_{1}} A_{\theta}|^2\ge0.
\end{equation*}
\end{lemma}
\begin{proof}
Since $e_{1}$ is a principal direction with curvature $1$. Thus $\bar{\nabla}_{e_{1}}\nu=h_{1i}e_{i}=e_{1}$. By calculating directly, we have
\begin{equation*}%\label{sxsxcd}
  \bar{\nabla}_{e_{1}}A_{\theta}=\bar{\nabla}_{e_{1}}(V-\cos\theta \bar{g}(x,\nu))=-\langle L,\tau\rangle,
\end{equation*}
where $L\coloneqq A_{\theta}E_{n+1}+\cos\theta(V-\bar{\nabla}_{\nu}V)x$ and $\tau=\frac{e_{1}}{x_{n+1}}$.

From $|\tau|^{2}=1$ and $\langle\tau,\nu_{\delta}\rangle=\bar{g}(e_{1},\nu)=0$, we see that
\begin{equation}\label{sxscdf}
  |\bar{\nabla}_{e_{1}}A_{\theta}|^{2}=\langle L,\tau\rangle^{2}\leq|L|^{2}-\langle L,\nu_{\delta}\rangle^{2}.
\end{equation}
Since
\begin{equation}\label{dccojhv}
\langle L,\nu_{\delta}\rangle=(V-\cos\theta \bar{g}(x,\nu))\langle E_{n+1},\nu_{\delta}\rangle+\cos\theta(V-\bar{\nabla}_{\nu}V)\langle x,\nu_{\delta}\rangle=(V-\bar{\nabla}_{\nu}V)
-A_{\theta}.
\end{equation}

At the point $x\in\Sigma$ under consideration, let
$\overline{\mathbb B_\rho(q)}$ be the supporting horoball tangent to $\Sigma$
at $x$ and containing $\widehat\Sigma$, by the supporting-horoball
argument in Lemma \ref{lemma-horo-umbilical-cap}. Here
$q=(q',\rho)=x-\rho\nu_\delta(x)$, so
\begin{equation*}%\label{uuuwwpa}
\langle\nu_\delta(x),E_{n+1}\rangle
=\frac{x_{n+1}}{\rho}-1.
\end{equation*}
Consequently,
\begin{equation*}%\label{dfcvmj}
\rho=\frac{x_{n+1}}{1+\langle\nu_\delta(x),E_{n+1}\rangle}
=\frac{1}{V-\bar\nabla_\nu V},
\end{equation*}
and
\begin{equation*}%\label{sxghnk}
  |q'|^{2}+\rho^{2}=|q|^{2}=|x-\rho\nu_{\delta}|^{2}=|x|^{2}-2\rho\langle x,\nu_{\delta}\rangle+\rho^{2},
\end{equation*}
which imply
\begin{equation}\label{sjxindsd}
  |x|^{2}=|q'|^{2}+2\rho\langle x,\nu_{\delta}\rangle.
\end{equation}
Combining \eqref{sxscdf}-\eqref{dccojhv} and \eqref{sjxindsd}, we have
\begin{eqnarray}\label{jxnsjdwyd}
2A_{\theta}(V-\bar{\nabla}_{\nu}V)-|\bar{\nabla}_{e_{1}} A_{\theta}|^2&\geq&2A_{\theta}(V-\bar{\nabla}_{\nu}V)-|L|^{2}+[(V-\bar{\nabla}_{\nu}V)
-A_{\theta}]^{2}\nonumber\\
&=&A^{2}_{\theta}+(V-\bar{\nabla}_{\nu}V)^{2}-|L|^{2}\nonumber\\
&=&A^{2}_{\theta}+\rho^{-2}-|A_{\theta}E_{n+1}+\cos\theta\rho^{-1}x|^{2}\nonumber\\
&=&\rho^{-2}(1-2\cos\theta\rho-\cos^{2}\theta|q'|^{2}),
\end{eqnarray}
where the last equality we used \eqref{sjxindsd}.
Since $E_{n+1}\in\operatorname{int}_{\mathcal H}(\widehat{\partial\Sigma}) \subseteq \overline{\mathbb{B}_{\rho}(q)}$, then 
\begin{equation}\label{xsccs}
  |q-E_{n+1}|^{2}=|q'|^{2}+(1-\rho)^{2}\leq\rho^{2}.
\end{equation}
Putting \eqref{xsccs} into \eqref{jxnsjdwyd}, we get
\begin{equation*}%\label{xsdsccxsxs}
 2A_{\theta}(V-\bar{\nabla}_{\nu}V)-|\bar{\nabla}_{e_{1}} A_{\theta}|^2\geq\rho^{-2}(1-2\cos\theta\rho-\cos^{2}\theta(2\rho-1))\geq\rho^{-2}(1-\cos\theta)^{2}>0,
\end{equation*}
where we used the fact $\cos\theta\geq0$ and $\rho=\frac{1}{V-\bar{\nabla}_{\nu}V}\leq\frac{1}{1+\cos\theta}$ from Lemma \ref{horo-bate}.
\end{proof}
Now we apply Theorem~\ref{thm:tensor-m.p.} to prove that flow \eqref{equationflow} preserves the horo-convexity.
\begin{theorem}\label{thm:horo-convex-preserving}
For $\theta\in(0,\pi/2]$. If the initial hypersurface $\Sigma_0$ is  horo-convex, then along the flow \eqref{equationflow}, the evolving hypersurface $\Sigma_t$ is horo-convex for $t>0$.
\end{theorem}

\begin{proof}
Let $S_{ij}\coloneqq h_{ij}-g_{ij}$. Horo-convexity is equivalent to the nonnegativity of $S_{ij}$.

We first claim that
\begin{equation*}%\label{bry-eigenvalue}
\left(\nabla_{\mu} S_{i j}\right) v^{i} v^{j}(x, t) \geq 0,
\end{equation*}
whenever $S\geq0$ at a boundary point and $v$ is a null vector of $S$. One can choose an orthonormal frame $\{e_1,\ldots,e_{n-1},\mu\}$ with principal directions.  
%By choosing an orthonormal frame with principal directions, we may assume that $g_{i j}=\delta_{i j}$ and $h_{i j}=\lambda_{i}\delta_{i j}$ is diagonal at $p_{0} \in \mathbb{B}^{n}$, thus
If $\{e_{\alpha}\}_{\alpha=1}^{n-1}$ is a null direction, then Proposition \ref{formula11} (4) and Corollary \ref{cor222} give
\begin{equation*}%\label{ealp}
\nabla_{\mu} S_{\alpha \alpha}= h_{\alpha\alpha;\mu}=\sum_{\gamma}\tilde{h}_{\alpha\gamma}(h_{\mu\mu}-h_{\alpha\alpha})\delta_{\alpha\gamma}\geq 0.
\end{equation*}
If $\mu$ is a null direction, then Proposition \ref{formula11} (4) and
\eqref{Fmumu} give
\begin{equation*}%\label{ealp2}
\nabla_{\mu} S_{\mu\mu} =h_{\mu\mu;\mu}=-\sum_{\alpha=1}^{n-1}\frac{F^{\alpha\alpha}}{F^{\mu\mu}}h_{\alpha\alpha;\mu}=-\sum_{\alpha=1}^{n-1}\frac{F^{\alpha\alpha}}{F^{\mu\mu}} \tilde h_{\alpha\alpha} (h_{\mu\mu}-h_{\alpha\alpha})\geq0.
\end{equation*}
Proposition \ref{formula11} (3)--(4) also gives $(\nabla_\mu S)_{\alpha\beta}=0$ for $\alpha\ne\beta$ in this frame. Moreover, if $\mu$ is a null direction, then $S(\mu,\mu)$ has a tangential minimum of zero and $(\nabla_\mu S)_{\alpha\mu}=(\nabla_\alpha S)_{\mu\mu}=0$ by Codazzi formula. Consequently, the claim is true.

%the boundary inequality holds for every null vector of $S$.

Next we consider the evolution equation of $S_{ij}$. From Proposition \ref{formula111} $(1)$ and $(4)$, we see
\begin{eqnarray}\label{eq06}
(\partial_t-\mathfrak{L}_{\mathcal{T}}) S_{i j}&=&-f_{;ij}+f((h^{2})_{ij}+\bar{g}_{ij})-2fh_{ij}\nonumber\\
&=&-\left(\frac{A_{\theta}}{F}-\bar{g}(X,\nu)\right)_{;ij}+f((h^{2})_{ij}+\bar{g}_{ij}-2h_{ij}),
\end{eqnarray}
where $$\mathfrak{L}_{\mathcal{T}}S_{ij}\coloneqq \nabla_{\mathcal{T}}S_{ij}+S_{kj}\nabla_{i}\mathcal{T}^{k}+S_{ki}\nabla_{j}\mathcal{T}^{k}.$$ By calculating directly, we get
\begin{eqnarray}\label{eq07}
\left(\frac{A_{\theta}}{F}\right)_{;ij}&=&\frac{2A_{\theta}}{F^{3}}F_{;i}F_{;j}-\frac{A_{\theta}}{F^{2}}(F^{kl;rs}h_{kl;i}h_{rs;j}+F^{kl}h_{kl;ij})\nonumber\\
&&-\frac{1}{F^{2}}(F_{;i}A_{\theta;j}+F_{;j}A_{\theta;i})+\frac{ \nabla_{ij}A_{\theta}}{F},
\end{eqnarray}
and
\begin{equation*}
(S^{2})_{ij}\coloneqq \bar{g}^{kl}S_{ik}S_{lj}=(h^{2})_{ij}+\bar{g}_{ij}-2h_{ij}.
\end{equation*}
By using \eqref{Fij2} we imply that
\begin{eqnarray}\label{eq08}
F^{kl}h_{kl;ij}&=&F^{kl}h_{ij;kl}+F^{kl}(h^{2})_{kl}h_{ij}-F^{kl}h_{kl}(h^{2})_{ij}-F^{kl}h_{kl}\bar{g}_{ij}+F^{kl}\bar{g}_{kl}h_{ij}\nonumber\\
&=&F^{kl}h_{ij;kl}+F^{kl}(h^{2})_{kl}h_{ij}-F(h^{2})_{ij}-F\bar{g}_{ij}+\mathcal{F}h_{ij},
\end{eqnarray}
where we used \eqref{F-110} and the fact $F^{i k}h_{kj}=F^{j k}h_{ki}$.

Putting \eqref{eq07}-\eqref{eq08} into \eqref{eq06}, and using \eqref{F113},\eqref{F119},\eqref{F119-X}, we see
\begin{eqnarray*}%\label{eq1009}
\partial_t S_{i j}
&=&A_{\theta}F^{-2}F^{kl}S_{ij;kl}+\bar{g}(X+\cos\theta F^{-1}x,\nabla S_{ij})+A_{\theta}F^{-2}F^{kl;rs}h_{kl;i}h_{rs;j}-2A_{\theta}F^{-3}F_{;i}F_{;j}\nonumber\\
&&+F^{-2}(F_{;i}A_{\theta;j}+F_{;j}A_{\theta;i}) +(S_{ik}+\bar{g}_{ik})(\bar{g}(\bar{\nabla}_{e_j}X,e_k)+\cos\theta F^{-1}\bar{g}(\bar{\nabla}_{e_j}x,e_k))\nonumber\\&&+(S_{jk}+\bar{g}_{jk})(\bar{g}(\bar{\nabla}_{e_i}X,e_k)+\cos\theta F^{-1}\bar{g}(\bar{\nabla}_{e_i}x,e_k))\nonumber\\
&&-((S^{2})_{ij}+2S_{ij}+\bar{g}_{ij})(\bar{g}(X,\nu)+VF^{-1})+\bar{g}_{ij}(\bar{g}(X,\nu)-2A_{\theta}F^{-1}-\bar{\nabla}_{\nu}V)\nonumber\\
&&+(S_{ij}+\bar{g}_{ij})(A_{\theta}F^{-2}F^{kl}(h^{2})_{kl}+A_{\theta}F^{-2}\mathcal{F}-V+\bar{\nabla}_{\nu}V F^{-1})+f(S^{2})_{ij}+\mathfrak{L}_{\mathcal{T}} S_{i j}.\nonumber\\
&\coloneqq &a^{kl}\nabla_l\nabla_k S_{ij}+ b^k\nabla_k S_{ij}+N_{ij}.\nonumber
\end{eqnarray*}

To use the tensor maximum principle, we next show that \eqref{tensor MP 1} holds when $S_{i j} \geq 0$ and $S_{i j} v^{j}=0$, where $v$ is a null vector of $S$. we take an orthonormal frame in which $h_{ij}$ is diagonal such that $1=\kappa_1\leq\kappa_2\leq\cdots\leq\kappa_n$
and the null vector is $e_1$. Hence, the terms containing $S_{ij}$ and $(S^2)_{ij}$ vanish on the null vector. Let $(x_0,t_0)$ be the point under consideration. 
%The terms containing $S_{ij}$ or $(S^2)_{ij}$ vanish when evaluated on this null vector. 
%Let $\left(x_{0}, t_{0}\right)$ be the point where $S_{i j}$ has a null vector $v$.  By continuity, we can assume that the principal curvatures are mutually distinct and in increasing order at $\left(x_{0}, t_{0}\right)$, that is $\kappa_{1}<\kappa_{2}<\cdots<\kappa_{n}$. The null vector condition $S_{i j} v^{j}=0$ implies that $v=e_{1}$ and $S_{11}=\kappa_{1}-1=0$ at $\left(x_{0}, t_{0}\right)$. The terms involving $S_{i j}$ and $\left(S^{2}\right)_{i j}$ satisfy the null vector condition and can be ignored. 
It remains to show that
\begin{eqnarray*}%\label{Q}
Q&\coloneqq &A_{\theta}F^{-2}F^{kl;rs}h_{kl;1}h_{rs;1}-2A_{\theta}F^{-3}|\nabla_{1}F|^{2}+2F^{-2}\nabla_{1}F\nabla_{1}A_{\theta}\nonumber\\
&&+A_{\theta}F^{-2}(F^{kl}((h^{2})_{kl}+\bar{g}_{kl})-2F)+(V-\bar{\nabla}_{\nu}V)(1-F^{-1})\nonumber\\
&&+\sup_{\Lambda_{k}^{p}}2A_{\theta}F^{-2}F^{kl}(2\Lambda^{p}_{k}\nabla_{l}S_{1p}-\Lambda_{k}^{p}\Lambda^{q}_{l}S_{pq})\geq0,
\end{eqnarray*}
at $(x_0,t_0)$. If $F=1$, then $h=g$ and $\nabla h=0$
at this point, so $Q=0$. Next we assume that $F>1$.
For every $p$ with $\kappa_p=1$, the non-negativity of $S$
and the boundary condition imply $\nabla_kS_{1p}=0$. Therefore, we have
%at $\left(x_{0}, t_{0}\right)$ for all matrix $\left(\Lambda_{k}^{p}\right)$. By assumption, $S_{11}=0$ and $\nabla_{k} S_{11}=0$ at $\left(x_{0}, t_{0}\right)$. We have 
$$
\begin{aligned}
{F}^{k l}\left(2 \Lambda_{k}^{p} \nabla_{l} S_{1 p}-\Lambda_{k}^{p} \Lambda_{l}^{q} S_{p q}\right) &=\sum_{k=1}^{n} \sum_{\kappa_p>1} \dot{F}^{k}\left(2 \Lambda_{k}^{p} \nabla_{k} S_{1 p}-\left(\Lambda_{k}^{p}\right)^{2} S_{p p}\right) \\
&=\sum_{k=1}^{n} \sum_{\kappa_p>1}\dot{F}^{k}\left[\frac{\left(\nabla_{k} S_{1 p}\right)^{2}}{S_{p p}}-\left(\Lambda_{k}^{p}-\frac{\nabla_{k} S_{1 p}}{S_{p p}}\right)^{2} S_{p p}\right].
\end{aligned}
$$ The supremum is attained by taking
$\Lambda_k^p=\nabla_kS_{1p}/S_{pp}$ for $\kappa_p>1$
and $\Lambda_k^p=0$ for $\kappa_p=1$.  %Then the supremum of the last line in \eqref{Q} is obtained by choosing the matrix $\Lambda_{k}^{p}=$ $\frac{\nabla_{k} S_{1 p}}{S_{p p}}$ for $p \geq 2, k \geq 1$ and $\Lambda_{k}^{1}=0$ for all $k$. 
It yields that
\begin{eqnarray} 
 Q&=&A_{\theta}F^{-2}\left({F}^{k l; r s}  h_{k l;1} h_{r s;1}+2\sum_{k=1}^{n} \sum_{\kappa_p>1}  \dot{F}^{k} \frac{\left(\nabla_{k} S_{1 p}\right)^{2}}{S_{p p}}\right) \notag \\
&&-2A_{\theta}{F^{-3}}\left|\nabla_{1}F\right|^{2}+{2}{F^{-2}} \nabla_{1} A_{\theta} \nabla_{1} F\nonumber\\
&&+A_{\theta}F^{-2}\left(F^{kl}((h^{2})_{kl}+\bar{g}_{kl})-2F\right)+\left(V-\bar{\nabla}_{\nu}V\right)(1-F^{-1}).\label{eq1118}
\end{eqnarray}
By Codazzi formula, $B_{ip}\coloneqq\nabla_1h_{ip}=\nabla_iS_{1p}$,
which vanishes if $\kappa_p=1$.
Applying Proposition \ref{lem:shifted-inverse} at
$(\kappa_1+\varepsilon,\ldots,\kappa_n+\varepsilon)$
with this fixed $B$ and let $\varepsilon\to0$.
Since $F>1$, this means 
%By Gauss-Codazzi equation we have $\nabla_{k} S_{1 p}=\nabla_{1} h_{k p}$. Therefore, we choose $B_{ij}=\nabla_{1} h_{ij}$ for Lemma \ref{lem:shifted-inverse} to get
\begin{equation}\label{eq:inverse-matrix-oiuhb}
 F^{kl;rs}{h_{kl;1}}h_{rs;1}
 +2\sum_{k=1}^{n} \sum_{\kappa_p>1}  \dot{F}^{k} \frac{\left(\nabla_{k} S_{1 p}\right)^{2}}{S_{p p}}
 \ge\frac{2}{F-1}\left(\sum_i\dot{F}^i\nabla_{1}h_{ii}\right)^2=\frac{2}{F-1}|\nabla_{1}F|^{2}.
\end{equation}
Inserting \eqref{eq:inverse-matrix-oiuhb} into \eqref{eq1118} we have
\begin{eqnarray*}
 Q&\geq&2A_{\theta}{F^{-3}}(F-1)^{-1}\left|\nabla_{1}F\right|^{2}+{2}{F^{-2}} \nabla_{1} A_{\theta} \nabla_{1} F\nonumber\\
&&+A_{\theta}F^{-2}\left(F^{kl} \left((h^{2})_{kl}+\bar{g}_{kl} \right)-2F\right)+\left(V-\bar{\nabla}_{\nu}V\right)(1-F^{-1})\nonumber\\
&=&2A_{\theta}{F^{-3}}(F-1)^{-1}\left|\nabla_{1}F+(2A_{\theta})^{-1}F(F-1)\nabla_{1}A_{\theta}\right|^{2}\nonumber\\
&&+A_{\theta}F^{-2}\left(F^{kl} \left((h^{2})_{kl}+\bar{g}_{kl} \right)-2F\right)+\left(V-\bar{\nabla}_{\nu}V-(2A_{\theta})^{-1}|\nabla_{1}A_{\theta}|^{2}\right)(1-F^{-1}).\nonumber
\end{eqnarray*}
By directly calculating, we see that ${F}^{k l}\left(\left(h^{2}\right)_{k l}+\bar{g}_{k l}\right)-2 F=\sum_{k=1}^{n} \dot{F}^{k}\left(\kappa_{k}-1\right)^{2} \geq 0$.
From Lemma \ref{lem:geometric-gradient}, we conclude that $Q\geq0$. Then $S_{i j}\geq0$ for any $t>0$ from the tensor maximum principle (Theorem \ref{thm:tensor-m.p.}). Therefore, the horo-convexity is preserved along $\Sigma_{t}$. 
\end{proof}

We now turn to strictly convex initial data without the
horo-convexity assumption. For $F=H_n/H_{n-1}$, the identity in \eqref{F-111} removes the last term in \eqref{F88}, and its remaining zero-order term is nonnegative. This gives the following lower bound for $F$.

\begin{prop}  Along flow \eqref{equationflow} with a strictly convex initial hypersurface and $F=\frac{H_n}{H_{n-1}}$, then the hypersurfaces $\Sigma_t$ satisfy 
\begin{equation*}%\label{C0-est}
 F(p, t) \geq  \min _{\overline{\mathbb{B}}^{n}} F(\cdot, 0), \quad \forall(p, t) \in \overline{\mathbb{B}}^{n} \times\left[0, T^{*}\right) .
\end{equation*}
\end{prop}
\begin{proof}From \eqref{E-yghs}, \eqref{F-111} and Proposition \ref{Festimate}, we know
\begin{equation*}%\label{C0-est2}
\mathcal{L}F \geq 0, \quad \bmod ~\nabla F.
\end{equation*}
Taking into account the boundary condition $\nabla_{\mu}F=0$ by \eqref{boundaryF}, the assertion follows from the maximum principle.
\end{proof}
As a direct consequence, we get the uniform curvature positive lower bound.% by Proposition \ref{Hk}.
\begin{corollary}\label{convex-prev}
Along flow \eqref{equationflow} with a strictly convex initial hypersurface and $F=\frac{H_n}{H_{n-1}}$, then the hypersurfaces $\Sigma_t$ are uniformly convex for $t\in[0,T^*)$, that is, there exists $c>0$ depending only on $\Sigma_0$ such that 
\begin{equation*}%\label{kappa}
  \min _{1\leq i\leq n} \kappa_{i }(p, t) \geq c, \quad \forall(p, t) \in \overline{\mathbb{B}}^{n} \times\left[0, T^{*}\right) .
\end{equation*}
\end{corollary}
\begin{proof}
Since
\begin{equation*}
 F=\frac{H_n}{H_{n-1}}
 =\frac{n}{\sum_{i=1}^n\kappa_i^{-1}}
 \leq n\min_{1\leq i\leq n}\kappa_i,
\end{equation*}
the assertion follows from the preceding lower bound for $F$.
\end{proof}
Thus, in both settings, the principal curvatures remain positive
and $F$ has a uniform positive lower bound. We next prove upper
bounds for $F$ and $H$.

\subsection{Upper curvature bounds}\ 

We first bound $F$ from above by using the negative term $-|X|^2$ in the evolution equation \eqref{Phsi} for the function $\psi$.

\begin{prop}%\label{F-upper}
Along flow \eqref{equationflow}, there exists $C>0$ depending only on $\Sigma_{0}$ such that
\begin{equation*}%\label{Fbounded}
  F(p, t) \leq C, \quad \forall(p, t) \in \overline{\mathbb{B}}^{n} \times\left[0, T^{*}\right).
\end{equation*}
\end{prop}
\begin{proof}
Recall that $\psi=\frac12|X|^2=\frac12\rho^2$.
From \eqref{C0}, we have 
\begin{equation}\label{eq:rho-bounds}
 (1-\cos\theta)R_1\leq\rho\leq R_2,
\end{equation}
and
\begin{equation*}%\label{eq:psi-derivative-bounds}
 |\psi_\nu|+|\bar\nabla^2\psi|\le C_0,
 \qquad
 |\nabla V|\le C_0,
 \qquad
\mathcal F\leq k,
\end{equation*}
where $0<R_1<R_2<\infty$ are two constants and \(C_0=C_0(n,R_{1},R_2)\). 
Indeed, $\mathcal F=1$ for $k=1$, while for $k\geq2$,
\begin{equation*}
 \mathcal F=k-\frac{k(n-k+2)}{n-k+1}
 \frac{\sigma_k\sigma_{k-2}}{\sigma_{k-1}^2}\leq k.
\end{equation*}

Consider the function
\begin{equation*}%\label{auxillary}
  \Phi\coloneqq \log F-V+\psi.
\end{equation*}
Using \eqref{factor},\eqref{boundarypsi},\eqref{boundaryF}, on $\partial \Sigma_{t}$,
\begin{equation*}%\label{logF}
 \nabla_{\mu} \Phi=F^{-1}\nabla_{\mu}F-\nabla_{\mu}V+\nabla_{\mu}\psi=-\csc\theta A_{\theta}<0.
 \end{equation*}
 Fix $T<T^*$. If the maximum of $\Phi$ on
$\overline{\mathbb B}^n\times[0,T]$ occurs at $t=0$,
the estimate follows from the initial data. Otherwise it occurs
at a positive-time interior spatial point $(p_0,t_0)$, where $p_{0} \in \mathbb{B}^{n}$ and $t_0>0$, then we have
\begin{equation}\label{eq90}
  \nabla \Phi\left(p_{0}\right)=0, \quad \nabla^{2} \Phi\left(p_{0}\right) \leq 0, \quad \partial_{t} \Phi\left(p_{0}\right) \geq 0.
\end{equation}
At this point, $F^{-1}F_{;i}=V_{;i}-\psi_{;i}$.
In a principal orthonormal frame,
$|V_{;i}-\psi_{;i}|\leq C$ and
$|A_{\theta;i}|\leq C(1+\kappa_i)$.
Using $\sum\limits_i\dot F^i\kappa_i=F$ and
\eqref{F88}, we obtain
\begin{eqnarray}
 \mathcal L(\log F)
 &=&2F^{-2}\sum_i\dot F^i A_{\theta;i}(V_{;i}-\psi_{;i})
 -A_\theta F^{-2}\sum_i\dot F^i(V_{;i}-\psi_{;i})^2\nonumber\\
 &&+F^{-1}(\mathcal F-1)\bar g(E_{n+1},\nu)
 +VF^{-2}(F^2-F^{ij}(h^2)_{ij})\nonumber\\
 &\leq&C_1F^{-1}+C_2F^{-2}.\label{logF1}
\end{eqnarray}
By \eqref{V-eqa} we deduce
\begin{eqnarray*}
-\mathcal{L}V&=&2A_{\theta}F^{-1}\bar{g}(\nu,E_{n+1})-\bar{g}(X,E_{n+1})+A_{\theta}VF^{-2}\mathcal{F}+\cos\theta F^{-1}\bar{g}(x,\nabla V)\\
&\leq&C_{1}F^{-1}+C_{2}F^{-2}.\nonumber
\end{eqnarray*}
From \eqref{Phsi} and \eqref{eq:rho-bounds} we see
\begin{eqnarray}
\mathcal{L}\psi&=&-|X|^{2}-F^{-2}[2x^{3}_{n+1}A_{\theta}F^{ij}\bar{g}(e_{i},E_{n+1})\bar{g}(e_{j},X)+x_{n+1}A_{\theta}\mathcal{F}] \notag \\
&&+F^{-1}[2A_{\theta}x^{2}_{n+1}\bar{g}(X,\nu)-\cos\theta x^{2}_{n+1}\bar{g}(X,x)+\cos\theta x^{2}_{n+1}\bar{g}(X,\nu)\bar{g}(x,\nu)]\nonumber\\
&\leq&-(1-\cos\theta)^2 R^{2}_{1}+C_{1}F^{-1}+C_{2}F^{-2}.\label{Phusg}
\end{eqnarray}
Combining \eqref{logF1}-\eqref{Phusg} with \eqref{eq90} we get
\begin{equation*}
  0\leq\mathcal{L}\Phi=\mathcal{L}(\log F)-\mathcal{L}V+\mathcal{L}\psi\leq-(1-\cos\theta)^2 R^{2}_{1}+C_{1}F^{-1}+C_{2}F^{-2}.\nonumber
\end{equation*}
%Thus we obtain the positive upper bound for $F$.
This bounds $F(p_0,t_0)$. Since $V$ and $\psi$ are uniformly
bounded, the maximality of $\Phi$ gives the same uniform upper
bound on the whole interval $[0,T]$. The constants are independent
of $T<T^*$.
\end{proof}

The upper and lower bounds for $F$ allow us to estimate $H$
using its evolution equation and the boundary inequality
\eqref{boudH}.
\begin{prop}\label{4.8}
 Along flow \eqref{equationflow}, there exists $C>0$ depending only on $\Sigma_{0}$ such that
\begin{equation}\label{ffff}
  H(p, t) \leq C, \quad \forall(p, t) \in \overline{\mathbb{B}}^{n} \times\left[0, T^{*}\right).
\end{equation}

\end{prop}

\begin{proof}Fix $T<T^*$. Unless the maximum of $H$ on
$\overline{\mathbb B}^n\times[0,T]$ occurs at $t=0$,
let $(p_0,t_0)$ be a maximum point of $H$ with $t_0>0$. It follows from \eqref{boudH} that $\nabla_{\mu} H \leq 0$ on $\partial \Sigma_{t}$, which rules out boundary maximum points, hence $p_{0} \in \mathbb{B}^{n}$ and there holds
%At a boundary maximum, maximality and \eqref{boudH} imply $\nabla_\mu H=0$. 
%Thus in either case
\begin{equation*}
 \nabla H=0,\qquad \nabla^2H\leq0,\qquad \partial_tH\geq0
 \quad\text{at }(p_0,t_0).
\end{equation*}
All subsequent computations are carried out at this point.

Recall the evolution equation \eqref{H1} of $H$ as follows:
\begin{eqnarray}\label{H108u}
0\leq\mathcal{L}H&=&A_{\theta}F^{-2}\bar{g}^{ij}F^{kl;rs}h_{kl;i}h_{rs;j}\nonumber\\ &&+\left[(n-F^{-1}H)\bar{g}(E_{n+1},\nu)+VH+A_{\theta}F^{-2}H F^{ij}(h^2)_{ij}+A_{\theta}F^{-2}H\mathcal{F}-nA_{\theta}F^{-1} \right]\nonumber\\%&&+[(n-F^{-1}H)\bar{g}(E_{n+1},\nu)+\left(A_{\theta}+V\right)H+A_{\theta}F^{-2}H\mathcal{F}-nA_{\theta}F^{-1}]\nonumber\\
&&+ \left[2F^{-2}\bar{g}(\nabla A_{\theta},\nabla F)-\left(A_{\theta}+V\right)|h|^{2}F^{-1}-2A_{\theta}F^{-3}|\nabla F|^{2} \right]\nonumber\\
&\coloneqq &J_{1}+J_{2}+J_{3}.
\end{eqnarray}
From the concavity of $F$ by Proposition \ref{garding-cone} and \eqref{A-yghs}, we see
\begin{equation}\label{J1-beg}
  J_{1}\coloneqq A_{\theta}F^{-2}\bar{g}^{i j} h_{k l ; i} h_{rs;j}F^{kl;rs} \leq 0.
\end{equation}
Using \eqref{F-111} and $1\leq\mathcal F\leq k$, and the positive lower bound for $F$, we obtain
\begin{equation}\label{J2-sok}
 J_2\leq C_1H+C_2,
\end{equation}
where $C_{1}$ and $C_{2}$ are some positive constants.

We now estimate $J_{3}$.
\begin{eqnarray*}
J_{3}&\coloneqq &2F^{-2}\bar{g}(\nabla A_{\theta},\nabla F)-\left(A_{\theta}+V\right)F^{-1}|h|^{2} -2A_{\theta}F^{-3}|\nabla F|^{2}\\
&=&-2A_{\theta}F^{-3}|\nabla F-(2A_{\theta})^{-1}F\nabla A_{\theta}|^{2}+(2A_{\theta})^{-1}F^{-1}|\nabla A_{\theta}|^{2}-\left(A_{\theta}+V\right)F^{-1}|h|^{2}\nonumber\\
&\leq&(2A_{\theta})^{-1}F^{-1}|\nabla A_{\theta}|^{2}-\left(A_{\theta}+V\right)F^{-1}|h|^{2}.\nonumber
\end{eqnarray*}
By choosing an orthonormal principal frame at $(p_0,t_0)$, we may assume that $g_{ij}=\delta_{ij}$ and $h_{ij}=\lambda_i\delta_{ij}$. Thus
\begin{eqnarray*}
|\nabla A_{\theta}|^{2}&=&|\nabla (V-\cos\theta\bar{g}(x,\nu))|^{2}
\leq\cos^{2}\theta\sum_{i=1}^{n}\lambda_{i} ^2\bar{g}(e_{i},x)^{2}+C(H+1)\\
&\leq&\cos^{2}\theta|h|^{2}(\bar{g}(x,x)-\bar{g}(x,\nu)^{2})+C(H+1).
\end{eqnarray*}
It follows that
\begin{eqnarray}
J_{3}&\leq&(2A_{\theta})^{-1}F^{-1}(|\nabla A_{\theta}|^{2}-2A_{\theta}(V+A_{\theta})|h|^{2}) \notag\\
&\leq&(2A_{\theta})^{-1}F^{-1}|h|^{2}[\cos^{2}\theta(\bar{g}(x,x)-\bar{g}(x,\nu)^{2})-2A_{\theta}(V+A_{\theta})]+C_{1}(H+1)\nonumber\\
&=&-(2A_{\theta})^{-1}F^{-1}|h|^{2}[3A^{2}_{\theta}-\cos^{2}\theta\bar{g}(x,x)+V^{2}]+C_{1}(H+1).\label{eqn:J-3-ineq}
\end{eqnarray}

Now we \textbf{claim}:
\begin{equation}\label{jsisj}
3A_{\theta}^{2}-\cos^{2}\theta\bar{g}(x,x)+V^{2}\geq\varepsilon_{0}>0.
\end{equation}
Indeed, by the containment condition \eqref{assump}
\begin{equation*}
  \Sigma_{t}\subseteq \widehat{C_{\theta,R_{2}}(E_{n+1})} \quad\text{and}\quad R_{2}\cdot\cos\theta<1,
\end{equation*}
which implies
\begin{equation}\label{sxgi}
  |x|\leq|x-(1-R_{2}\cos\theta)E_{n+1}|+(1-R_{2}\cos\theta)\leq R_{2}+(1-R_{2}\cos\theta)=1+R_{2}(1-\cos\theta),
\end{equation}
and
\begin{equation}\label{sxxs}
  m\coloneqq \cos\theta(1+R_{2}(1-\cos\theta))<\cos\theta+(1-\cos\theta)=1.
\end{equation}
Using \eqref{sxgi}-\eqref{sxxs} we get
\begin{equation*}%\label{xjsis}
  \cos^{2}\theta\bar{g}(x,x)=\cos^{2} \theta V^{2}|x|^{2}\leq m^{2}V^{2}<V^{2}.
\end{equation*}
Thus \textbf{claim} \eqref{jsisj} follows.

Inserting \eqref{jsisj} into \eqref{eqn:J-3-ineq}, we have
 \begin{equation*}%\label{J3-a}
   J_{3}\leq -C_{3}|h|^{2}+C_{1}(H+1).
 \end{equation*}
Together with \eqref{J1-beg} and \eqref{J2-sok}, it follows from \eqref{H108u} that
\begin{equation*}%\label{sdxs}
  0\leq\mathcal{L}H=J_{1}+J_{2}+J_{3}\leq-C_{3}|h|^{2}+C_{4}(H+1).
\end{equation*}
Using the simple inequality $n|h|^{2}\geq H^{2}$, we conclude that $H$ is bounded above.  
\end{proof}

Proposition \ref{4.8}, together with the positivity of the
principal curvatures provided by Theorem
\ref{thm:horo-convex-preserving} or Corollary
\ref{convex-prev}, gives the following upper curvature bound. 
\begin{corollary}\label{cor56}
 Along flow \eqref{equationflow}, $\Sigma_{t}$ for $t\in\left[0, T^{*}\right)$ has uniform curvature bound, that is, there exists $C>0$ depending only on $\Sigma_{0}$ such that the principal curvatures of $\Sigma_{t}$,
\begin{equation*}
\max _{1\leq i\leq n} \kappa_{i}(p, t) \leq C, \quad \forall(p, t) \in \overline{\mathbb{B}}^{n} \times\left[0, T^{*}\right).
\end{equation*}
\end{corollary}

\subsection{Long-time existence and smooth convergence} \ %\label{convergence-section}

We first use the scalar formulation in \eqref{scalarflow} and the curvature bounds to establish uniform regularity and long-time existence. Quermassintegral monotonicity then shows that every subsequential limit is an umbilical cap. Finally, a comparison with the model caps determines the position of the limiting cap and gives convergence of the whole flow.
\begin{prop}
The flow \eqref{equationflow} exists for all time and has uniform
$C^\infty$-estimates on $[t_0,\infty)$ for every $t_0>0$.
\end{prop}

\begin{proof}
The estimates from Theorem \ref{thm:horo-convex-preserving} and Corollaries \ref{convex-prev}-\ref{cor56} make equation \eqref{scalarflow} uniformly parabolic on a compact admissible curvature range, with a smooth uniformly oblique boundary condition. Therefore, we may apply the local boundary estimates of \cite[Theorem 14.22]{Lieb} (see also, \cite[Theorem 6.1, Theorem 6.4 and Theorem 6.5]{GD} and \cite[Theorem 5]{NN}). These imply uniform $C^{2+\alpha,1+\alpha/2}$ estimates on positive-time strips. Thus, parabolic Schauder estimates and continuation yield long-time existence and uniform $C^\infty$ estimates on $[t_0,\infty)$ for every $t_0>0$.
\end{proof}

\begin{prop}%\label{uniqueness-convergence}
The hypersurfaces $\Sigma_t$ converge smoothly, after
reparameterization, to a uniquely determined umbilical cap
centred at $E_{n+1}$ as $t\to\infty$.
\end{prop}

\begin{proof} For $k\geq2$, Proposition \ref{Wk-Wn} shows that $W_{1,\theta}(\widehat\Sigma_t)$ is non-decreasing. Precisely,
\begin{eqnarray}\label{eq:derivative_W1}
\partial_{t} W_{1, \theta}(\widehat{\Sigma_{t}})&=&\frac{n}{n+1}\int_{\Sigma_{t}}H_{1}\left(\frac{A_{\theta}H_{k-1}}{H_{k}}-\bar{g}(X,\nu)\right)dA\nonumber\\
&=&\frac{n}{n+1} \int_{\Sigma_{t}}A_{\theta}\left(\frac{H_{1} H_{k-1}}{H_{k}}-1\right)dA\geq0,
\end{eqnarray}
where we used \eqref{A-yghs} and \eqref{mink2} and Newton--MacLaurin inequality \eqref{Newton}.

The uniform estimates and compact barriers imply that
$W_{1,\theta}(\widehat\Sigma_t)$ is bounded above. Thus
\begin{equation*}
\int_0^\infty\partial_tW_{1,\theta}(\widehat\Sigma_t)~dt
=\lim_{t\to\infty}W_{1,\theta}(\widehat\Sigma_t)
-W_{1,\theta}(\widehat\Sigma_0)<\infty.
\end{equation*}
The nonnegative $\partial_{t}W_{1,\theta}(\widehat\Sigma_t)$
is integrable and uniformly continuous by the preceding estimates,
and therefore tends to zero. It follows from \eqref{eq:derivative_W1}, the lower bound
$A_\theta\geq c>0$, and the equality case of \eqref{Newton}
that every subsequential limit is an umbilical cap for $k\geq2$.

For $k=1$, we replace $W_{1,\theta}$ with $W_{2,\theta}$ in \eqref{eq:derivative_W1}. By
\eqref{eqn:lower-order-1vf} and \eqref{mink2}, we see
\begin{eqnarray*}
\partial_{t}W_{2,\theta}(\widehat\Sigma_t)
=\frac{n-1}{n+1}\int_{\Sigma_t}
A_\theta\left(\frac{H_2}{H_1}-H_1\right)dA\leq0.
\end{eqnarray*}
The uniform estimates bound $W_{2,\theta}$ from below and make
its derivative uniformly continuous. Hence $\partial_{t}W_{2,\theta}(\widehat\Sigma_t)$ also tends
to zero. Since $A_\theta>0$, the equality case
$H_1^2=H_2$ of \eqref{Newton} implies that every subsequential limit is again an umbilical cap.

The preservation property of $W_{k,\theta}$ and its strict monotonicity on strictly convex model caps determine the same radius $r_\infty$. Consequently, $F\to r_\infty^{-1}-\cos\theta$ uniformly.

Finally, the uniqueness of the convergent subsequence can be proved similarly in \cite[Proposition 3.15]{WengX} (see also \cite[Proposition 4.5]{Guo1}) by the barrier-radius argument. 
Roughly, the largest and smallest radii of centred comparison caps are
monotone and bounded. At a point realizing either extremal radius, from  \eqref{dfec} and \eqref{f} we have,
\begin{equation*}
 \frac{dr}{dt}
 =\frac{2x_{n+1}r^{2}\bar g(X,\nu)}
 {(r^{2}\sin^2\theta+|X|^{2})F}
 \cdot\left(\frac1r-\cos\theta-F\right).
\end{equation*}
The prefactor is bounded above and below by positive constants.
Since $F\to (r_\infty^{-1}-\cos\theta)$ uniformly as $t\to\infty$, a limiting extremal
radius different from $r_\infty$ would have derivative bounded
away from zero, contradicting the boundedness. Thus both extremal
radii tend to $r_\infty$, and smooth precompactness yields smooth
convergence to $C_{\theta,r_\infty}(E_{n+1})$ after reparameterization.
Hence we complete the proof. 
\end{proof}

\

\section{Alexandrov--Fenchel inequalities and a Gauss--Bonnet--Chern  formula}\label{sec5} 

The quermassintegral inequalities will be proved by combining the convergence of the flow with quermassintegral monotonicity. We first treat Theorem \ref{mainthm-horo} for horo-convex hypersurfaces, then deduce Theorem~\ref{mainthm} under the containment assumption, and finally evaluate $W_{n+1,\theta}$ through its deformation invariance.

\subsection{Proof of Theorem \ref{mainthm-horo}}

For $0<r<\frac{1}{\cos\theta}$, we consider the flow $x: \overline{\mathbb{B}}^{n} \times[0,+\infty) \rightarrow \mathcal{H}_{+}$ given by
\begin{eqnarray*}
\partial_{t} x&=&x-E_{n+1}\,\,\, \qquad\quad \text { on } \overline{\mathbb{B}}^{n} \times[0,+\infty)\nonumber \\
x(\cdot, 0)&=&C_{\theta, r_{0}}(E_{n+1}) \,\,\,\quad\quad \text { on } \overline{\mathbb{B}}^{n}.\nonumber
\end{eqnarray*}
The flow hypersurfaces of this flow are $C_{\theta, r(t)}(E_{n+1})$. By direct calculation, we see
\begin{equation*}
  r'(t)=r(t)>0 \,\,\,\,\quad \text{with} \,\,\,r(0)=r_{0}>0.
\end{equation*}
Hence from Theorem \ref{mainthm0.1}, we have
\begin{equation*}
r'(t)\frac{d}{dr}f_{k}(r)=\partial_{t} W_{k,\theta}\left(\widehat{C_{\theta, r(t)}(E_{n+1})}\right)=\frac{n+1-k}{n+1}\int_{C_{\theta, r(t)}(E_{n+1})} H_{k}\bar{g}(X,\nu)\, dA_{t}>0,
\end{equation*}
where the last inequality follows from the strict convexity of the umbilical cap; hence, it yields that $f_{k}(r)$ are strictly increasing functions with respect to $r$. 

For a strictly convex hypersurface $\Sigma$, consider the
contractions
\begin{equation}\label{contractions}
y(p,t)=E_{n+1}+t\left(x(p)-E_{n+1}\right),
\quad (p,t)\in\Sigma\times(0,1].
\end{equation}
Denote their images by $\Sigma_t$. They preserve the contact
angle and strict convexity by \eqref{conformal-relation} and
Proposition \ref{bounded}.  Their outward normal speed is
\begin{equation*}
 \bar g(\partial_t y,\nu)= \bar g(t^{-1}(y-E_{n+1}),\nu)
 =\frac1t\bar g(X,\nu)>0.
\end{equation*}
The definitions of the quermassintegrals give
$W_{l,\theta}(\widehat\Sigma_t)=O(t^{n+1-l})$ for
$0\leq l\leq n$. Hence, for $1\leq l\leq n$,
Theorem \ref{mainthm0.1} implies
\begin{equation*}
 W_{l,\theta}(\widehat\Sigma)
 =\frac{n+1-l}{n+1}
   \int_0^1\frac1t\int_{\Sigma_t}
      H_l\bar g(X,\nu)\,dA_t\,dt>0.
\end{equation*}
Therefore, $W_{l,\theta}(\widehat\Sigma)$ lies in the image of $f_l$,
so $f_l^{-1}\bigl(W_{l,\theta}(\widehat\Sigma)\bigr)$ is well defined.

By evolving along the flow \eqref{equationflow} with initial horo-convex hypersurface $\Sigma$, we denote its convergence limit by the umbilical cap  $C_{\theta, r_{\infty}}(E_{n+1})$. The barriers imply $r_\infty\leq R_2<1/\cos\theta$, with
the usual convention when $\theta=\pi/2$. By
Proposition \ref{Wk-Wn},
\begin{equation*}
 0<W_{l,\theta}(\widehat\Sigma)
 \leq f_l(r_\infty).
\end{equation*}
Since $f_l$ is continuous and strictly increasing and
$f_l(r)\to0$ as $r\to0$, this proves that
$W_{l,\theta}(\widehat\Sigma)$ lies in the image of $f_l$.  Hence, using Proposition \ref{Wk-Wn} again for $1\leq l<k\leq n$,
\begin{eqnarray*}%\label{horo-quesoiuj}
 W_{k, \theta}(\widehat{\Sigma})&=&W_{k, \theta}\left(\widehat{C_{\theta, r_{\infty}}(E_{n+1})}\right)=f_{k}\left(r_{\infty}\right) \\
&=&f_{k}\circ f_{l}^{-1} \circ f_{l}\left(r_{\infty}\right) \geq f_{k} \circ f_{l}^{-1}\left(W_{l, \theta}(\widehat{\Sigma})\right) .
\end{eqnarray*}

If equality holds, preservation of $W_{k,\theta}$ and strict
monotonicity of the model functions imply
$W_{l,\theta}(\widehat\Sigma)=f_l(r_\infty)$.
Since $W_{l,\theta}$ is non-decreasing, its derivative vanishes
for every positive time. By the proof of Proposition
\ref{Wk-Wn},
\begin{equation*}
 \int_{\Sigma_t}A_\theta
 \left(\frac{H_lH_{k-1}}{H_k}-H_{l-1}\right)dA_t=0.
\end{equation*}
The integrand is nonnegative and $A_\theta>0$. The equality
case of \eqref{Newton} therefore makes $\Sigma_t$ umbilical.
Passing to the initial time shows that $\Sigma$ is umbilical.
Conversely, every umbilical cap with the prescribed contact
angle has the same quermassintegrals as the centred model cap
of the same radius, by horizontal translation invariance,
and hence attains equality.

\qed

\

\subsection{Proof of Theorem \ref{mainthm}}

The model functions $f_{k}(r)$ and their strict monotonicity have been established in the proof of Theorem \ref{mainthm-horo}.

Suppose first that $\Sigma$ is strictly convex. Evolving $\Sigma$ by the flow \eqref{equationflow} with $k=n$, we denote its convergence limit by the umbilical cap given by $C_{\theta, r_{\infty}}(E_{n+1})$, then, from Proposition \ref{Wk-Wn} for $k=n$, we have
\begin{eqnarray*}
 W_{n, \theta}(\widehat{\Sigma})&=&W_{n, \theta}\left(\widehat{C_{\theta, r_{\infty}}(E_{n+1})}\right)=f_{n}\left(r_{\infty}\right) \\
&=&f_{n}\circ f_{l}^{-1} \circ f_{l}\left(r_{\infty}\right) \geq f_{n} \circ f_{l}^{-1}\left(W_{l, \theta}(\widehat{\Sigma})\right) .
\end{eqnarray*}
The equality argument in the proof of Theorem
\ref{mainthm-horo}, applied with $k=n$, shows that equality holds precisely for umbilical caps with the prescribed contact angle.

We now consider the weakly convex case. The smallest centred model
cap enclosing $\widehat\Sigma$ has radius at most $R$.
At a contact point, comparison gives
$h\geq(R^{-1}-\cos\theta)g>0$. If contact occurs at the
boundary, the common contact angle gives the same tangent plane;
the second derivative comparison on its inward half extends
to the whole tangent plane because a quadratic form is even.
Thus $\Sigma$ has a strictly convex interior point.

Under the contractions \eqref{contractions}, the Euclidean normal
is unchanged and \eqref{conformal-relation} gives
\begin{equation*}
 \kappa_i(\cdot,t)
 =\frac{1+t(x_{n+1}-1)}{t x_{n+1}}\kappa_i
 +\frac{1-t}{t x_{n+1}}
   \langle E_{n+1},\nu_\delta\rangle,
 \qquad 0<t<1.
\end{equation*}

For $\theta\in(0, \pi/2)$, we have $\langle E_{n+1},\nu_\delta\rangle\geq\cos\theta>0$ by the proof of Proposition \ref{bounded}. Hence these contractions are convex and retain a strictly convex
interior point. The same first variation argument as above gives
$W_{l,\theta}(\widehat\Sigma)>0$ for $1\leq l\leq n$.
Moreover, the displayed formula also makes the
contractions strictly convex for $t<1$, providing smooth
approximations as $t\to1$.

For $\theta=\pi/2$, we construct strictly convex approximations
that preserve the free boundary condition. Since $E_{n+1}\notin\Sigma$, we consider the Euclidean inversion
\begin{eqnarray*}
I(x)\coloneqq E_{n+1}
     +\frac{x-E_{n+1}}{|x-E_{n+1}|^2},
\end{eqnarray*}
where $|\cdot|$ denotes the Euclidean norm. This map is smooth
near $\Sigma$, preserves $\mathcal H=\{x_{n+1}=1\}$, and is conformal. Hence
$\widetilde\Sigma:=I(\Sigma)$ also meets $\mathcal H$
orthogonally.

Let $\widetilde x=I\circ x$, and choose the Euclidean unit normal
$\widetilde\nu_\delta$ corresponding to the original outward
normal under $I$. Define
\begin{eqnarray*}
\widetilde x_s
=\widetilde x-s\widetilde\nu_\delta,
\qquad
x_s=I\circ\widetilde x_s.
\end{eqnarray*}
For sufficiently small $|s|$, the Euclidean parallel hypersurfaces
$\widetilde x_s$ have the same unit normal
$\widetilde\nu_\delta$. Along the boundary, this normal is tangent
to $\mathcal H$. Thus their boundaries remain in $\mathcal H$,
and they continue to meet $\mathcal H$ orthogonally.
The same properties hold after inversion.

Compactness and transversality to $\mathcal H$ ensure that, after
reducing the parameter interval, $\Sigma_s:=x_s(\Sigma)$ remains
properly embedded in $\mathcal H_+$. Moreover, since
$I\circ I=\operatorname{Id}$,
\begin{eqnarray*}
x_0=x,
\qquad
x_s\to x
\quad\text{in }C^\infty\text{ as }s\to0.
\end{eqnarray*}

Differentiating the construction gives the normal variation
\begin{eqnarray*}
\partial_s x_s=f\nu,
\qquad
f=-\frac{|X|^2}{x_{n+1}},
\qquad
\mathcal T=0.
\end{eqnarray*}
Writing $d_{\mathbb H}(x,E_{n+1})$ for the hyperbolic distance,
we have
\begin{eqnarray*}
f=2-2\cosh d_{\mathbb H}(x,E_{n+1}),
\qquad
\bar\nabla^2f=(f-2)\bar g.
\end{eqnarray*}
From Proposition \ref{formula111} (5) we have
\begin{eqnarray*}
\partial_s h_j^{\,i}
=
2\delta_j^{\,i}
+(\bar\nabla_\nu f)h_j^{\,i}
-fh_k^{\,i}h_j^{\,k}.
\end{eqnarray*}
Thus each principal curvature satisfies
\begin{eqnarray*}
\partial_s\kappa_i
=
2+(\bar\nabla_\nu f)\kappa_i-f\kappa_i^2.
\end{eqnarray*}

Since $\kappa_i(0)\ge0$ and the right-hand side equals $2$ when
$\kappa_i=0$, the principal curvatures remain nonnegative.
The coefficients are uniformly bounded on a sufficiently small
parameter interval. We choose $C>0$ such that
$|\bar\nabla_\nu f|\le C$. Since $f\le0$, we obtain
\begin{eqnarray*}
\partial_s\kappa_i\ge2-C\kappa_i.
\end{eqnarray*}
Scalar comparison therefore gives
\begin{eqnarray*}
\kappa_i(s)
\ge e^{-Cs}\kappa_i(0)
   +\frac{2}{C}(1-e^{-Cs})
>0
\qquad\text{for }s>0.
\end{eqnarray*}
Hence $\Sigma_s$ is strictly convex for all sufficiently small
positive $s$ and converges smoothly to $\Sigma$, while preserving
the free boundary condition.

%For $\theta=\pi/2$, invert in the Euclidean unit sphere centred at $E_{n+1}$, take outward Euclidean parallel hypersurfaces, and invert back. Both operations preserve orthogonality to $\mathcal H$; small displacements preserve proper embeddedness.
%The resulting deformation has $\mathcal T=0$ and normal speed
%\begin{equation*}
% f=-\frac{|X|^2}{x_{n+1}}.
%\end{equation*}
%Since $\bar\nabla^2f=(f-2)\bar g$, Proposition
%\ref{formula111} (5) yields
%\begin{equation*}
% \partial_t h_j^i
% =2\delta_j^i+(\bar\nabla_\nu f)h_j^i
%   -f h_k^i h_j^k.
%\end{equation*}
%Its coefficients are uniformly bounded for small time, and the derivative of each principal curvature is $2$ when that curvature is zero. Scalar comparison therefore gives strict convexity for small positive time, including at the boundary.

In either case, the approximations retain
$E_{n+1}\in\operatorname{int}_{\mathcal H}
(\widehat{\partial\Sigma})$ and satisfy \eqref{assump} with
a common radius $R_2>R$ such that $R_2\cos\theta<1$.
Applying the strictly convex case and its barriers, then passing
to the limit, we obtain
\begin{equation*}
 0<W_{l,\theta}(\widehat\Sigma)\leq f_l(R_2).
\end{equation*}
Thus the inverse of $f_l$ is well defined at
$W_{l,\theta}(\widehat\Sigma)$ and is continuous. %Consequently, continuity of the quermassintegrals proves the inequality for $\Sigma$.
Consequently, continuity of the quermassintegrals allows us to
pass to the limit in the inequality for $\Sigma_s$,
yielding
\begin{equation}\label{AF-ineq-case2}
W_{n,\theta}(\widehat\Sigma)
\ge
f_n\circ f_l^{-1}
\bigl(W_{l,\theta}(\widehat\Sigma)\bigr).
\end{equation}

For the equality case, we use the local variational argument of
\cite[Section 4]{SWX}, with the first variation formula
\eqref{eqn:lower-order-1vf} for the present setting. Denote
\begin{equation*}
 r\coloneqq f_l^{-1}(W_{l,\theta}(\widehat\Sigma)).
\end{equation*}
Then $0<r\le R_2$, so $r^{-1}-\cos\theta>0$. Define
\begin{eqnarray*}
\Sigma_+
\coloneqq 
\{p\in\operatorname{int}\Sigma:h(p)>0\}.
\end{eqnarray*}
The preceding contact argument shows that $\Sigma_+$ is nonempty.

Let $\varphi\in C_c^\infty(\Sigma_+)$, and consider a normal
variation $\Sigma_s$ whose normal speed at $s=0$ is $\varphi$.
For sufficiently small positive and negative $s$, this variation
preserves convexity, embeddedness, and all boundary conditions.
The containment condition is also preserved by using the
outer radius $R_2>R$.

The difference between the two sides of the above inequality \eqref{AF-ineq-case2}, with $\Sigma$ replaced by $\Sigma_s$, is
nonnegative along this variation and vanishes at $s=0$.
Differentiating at this minimum and applying
Theorem~\ref{mainthm0.1}, we obtain
\begin{eqnarray*}
0=
\frac{1}{n+1}
\int_{\Sigma_+}
\left(
H_n-(n+1-l)\frac{f_n'(r)}{f_l'(r)}H_l
\right)\varphi dA.
\end{eqnarray*}
Since $\varphi$ is arbitrary,
\begin{eqnarray*}
H_n=(n+1-l)\frac{f_n'(r)}{f_l'(r)}H_l=\left(\frac1r-\cos\theta\right)^{n-l}H_l
\quad\text{on }\Sigma_+.
\end{eqnarray*}
The second equality follows from
the model cap first variation formula. By the Newton--MacLaurin
inequality $H_l\geq H_n^{l/n}$, we conclude that
\begin{equation*}
 H_n\geq\left(\frac1r-\cos\theta\right)^n>0\quad\text{on }\Sigma_+.
\end{equation*}
The principal curvatures of the fixed smooth hypersurface are
bounded above. Thus they have a uniform positive lower bound on $\Sigma_+$. This set $\Sigma_+$ is nonempty, open
and closed in $\operatorname{int}\Sigma$. Connectedness and
continuity up to the boundary show that $\Sigma$ is strictly
convex. The equality characterization in the strictly convex
case therefore applies.

Finally, we turn to prove the assertion \eqref{eq:top-order-GB}.
The geometric argument in the proof of Proposition \ref{bounded}
shows that $\Sigma$ is also a Euclidean convex capillary hypersurface in
the Euclidean half-space $\{x_{n+1}\geq1\}$. Let $x:\Sigma\rightarrow\mathcal H_+$ be the given embedding. We consider the same contraction flow as in \eqref{contractions}.
It then follows from Theorem \ref{mainthm0.1} that
\begin{equation*}
W_{n+1,\theta}(\widehat\Sigma)
=W_{n+1,\theta}(\widehat\Sigma_t),\qquad t\in(0,1].
\end{equation*}
By the Euclidean scaling and \eqref{quermassinterga-Euclidean}, we get
\begin{equation*}
W_l^{\mathbb R^n}(\widehat{\partial\Sigma_t})
=t^{n-l}W_l^{\mathbb R^n}(\widehat{\partial\Sigma}),
\qquad 0\leq l\leq n-1.
\end{equation*}
Moreover, the definitions \eqref{Wk-formula} give
$W_{n-1,\theta}(\widehat\Sigma_t)=O(t^{2})$. The conformal curvature relation shows that
$\int_{\Sigma_t}H_n dA_t$ tends to the Euclidean total
Gauss curvature of $\Sigma$, which equals the area of the spherical cap with polar angle $\theta$, see \cite[Eq. (1.13)]{WWX1}. Therefore, we have
\begin{equation*}
\lim_{t\rightarrow0}\int_{\Sigma_t}H_n dA_t
=\omega_{n-1}\int_0^\theta\sin^{n-1}s~ds.
\end{equation*}
We obtain equality \eqref{eq:top-order-GB} by taking $t\rightarrow0$ in \eqref{Wn-1113}. 
\qed

\bigskip 

\noindent{\it Acknowledgements.} J. Guo is supported by NSFC (Grant No.12501062). The authors used AI tools to assist with language editing and to identify points requiring mathematical and bibliographic
verification. The authors checked all mathematical arguments, references and revisions and take full responsibility for the
final manuscript.

\

\end{document}